\documentclass[a4paper,12pt]{amsart}

\usepackage{amsmath}
\usepackage{amssymb}
\usepackage{amsthm}
\usepackage{verbatim}
\usepackage[all]{xy}
\usepackage{enumerate}

\newtheorem{thm}{Theorem}[section]
\newtheorem{prop}[thm]{Proposition}
\newtheorem{cor}[thm]{Corollary}
\newtheorem{lem}[thm]{Lemma}

\theoremstyle{definition}

\newtheorem{dfn}[thm]{Definition}

\newtheorem{example}[thm]{Example}
\newtheorem{rmk}[thm]{Remark}
\newtheorem{nt}[thm]{Notation}

\numberwithin{equation}{section}
\newcommand{\cH}{\mathcal{H}}
\newcommand{\nphi}{\mathfrak{n}_\varphi}
\newcommand{\mphi}{\mathfrak{m}_\varphi}

\newcommand{\Dom}{\textrm{Dom}}

\newcommand{\vertt}{\vert\!\vert \! \vert }

\newcommand{\sS}{\mathcal{S}}

\title{The $L^p$-Fourier transform on locally compact quantum groups}
\author{Martijn Caspers}

\address{Radboud Universiteit Nijmegen, IMAPP, FNWI, Heyendaalseweg 135, 6525 AJ Nijmegen,
the Netherlands}
\email{caspers@math.ru.nl}

\date{ \noindent May 10th, 2011. \\
{\it Keywords:} Locally compact quantum groups, Fourier transform, Interpolation spaces. \\
{\it 2000 Mathematics Subject Classification numbers:}  20G42, 20N99, 47N99.
}

\begin{document}
 
\begin{abstract} Using interpolation properties of non-commutative $L^p$-spaces associated with an arbitrary von Neumann algebra, we define a $L^p$-Fourier transform $1 \leq p \leq 2$ on locally compact quantum groups. We show that the Fourier transform determines a distinguished choice for the interpolation parameter as introduced by Izumi. We define a convolution product in the $L^p$-setting and show that the Fourier transform turns the convolution product into a product.
\end{abstract}

\maketitle

\section{Introduction}

The Fourier transform is one of the most powerful tools coming from abstract harmonic analysis.  Many classical applications, in particular in the direction of $L^p$-spaces, can be found in for example \cite{Gra}.  Here we extend this tool by giving a definition of a Fourier transform on the non-commutative $L^p$-spaces associated with a locally compact quantum group. This gives a link between quantum groups and non-commutative measure theory.

\vspace{0.3cm}

Recall that the Fourier transform on locally compact abelian groups can be defined in an $L^p$-setting for $p$ any real number between 1 and 2. This is done in the following way. Let $G$ be a locally compact group and let $\hat{G}$ be its Pontrjagin dual. For a $L^1$-function $f$ on $G$, we define its Fourier transform $\hat{f}$ to be the function on $\hat{G}$, which is defined by 
\begin{equation}\label{EqnIntroFT}
\hat{f}(\pi) = \int f(x) \pi(x) d_lx, \qquad \pi \in \hat{G}.
\end{equation}
Then $\hat{f}$ is a continuous function on $\hat{G}$ vanishing at infinity. So we can consider this transform as a bounded map $\mathcal{F}_1: L^1(G) \rightarrow L^\infty(\hat{G})$. 
The Plancherel theorem yields that if $f$ is moreover a $L^2$-function on $G$, then $\hat{f}$ is a $L^2$-function on $\hat{G}$ and this map extends to a unitary map $\mathcal{F}_2: L^2(G) \rightarrow L^2(\hat{G})$.

It is known that the Fourier transform can be generalized in a $L^p$-setting by means of the Riesz-Thorin theorem, see \cite{BerghLof}. The statement of this theorem directly implies the following.  For any $p$, with  $1 \leq p \leq 2$, the linear map 
$
L^1(G) \cap L^2(G) \rightarrow L^2(\hat{G}) \cap L^\infty(\hat{G}): f \mapsto \hat{f}
$
 extends uniquely to a bounded map
\[
\mathcal{F}_p: L^p(G) \rightarrow L^q(\hat{G}), \qquad  \frac{1}{p} + \frac{1}{q} = 1.
\]  
This map $\mathcal{F}_p$ is known as the $L^p$-Fourier transform. 
 
\vspace{0.3cm}

Quantum groups have been around for quite some time and have been studied in many different guises. From the 80'ies onwards quantum groups are studied in an operator algebraic approach. In particular, a satisfactory C$^\ast$-algebraic definition of a compact  quantum group has been given by Woronowicz \cite{Wor}. 

Around 2000 locally compact quantum groups were introduced by Kustermans and Vaes \cite{KusV}, \cite{KusVII}, see also \cite{DaeleLcqg}. Their definitions include a C$^\ast$-algebraic one and a von Neumann algebraic one. Since their introduction many aspects of abstract harmonic analysis have been given a suitable extension for quantum groups. In particular, the Pontrjagin duality theorem has been generalized to locally compact quantum groups. In fact, this was one of the main motivations for their definition. So every locally compact quantum group admits a (Pontrjagin) dual quantum group such that the double dual is isomorphic to the originial quantum group.  

\vspace{0.3cm}

Since we now have a von Neumann algebraic interpretation for quantum groups at hand, it is natural to ask if the $L^p$-Fourier transform can be defined in this context. This is for two reasons. First of all, this framework studies quantum groups in a measurable setting which appeals to a more general interest: what links can be found between on one hand non-commutative measure spaces, in particular non-commutative $L^p$-spaces, and on the other hand the  theory of quantum groups.  The $L^p$-Fourier transform studied in the present paper establishes such a link. 
Secondly, the existence of a Pontrjagin dual is always guaranteed in the Kustermans-Vaes setting. This is an essential ingredient for defining Fourier transforms. 

\vspace{0.3cm}

The $L^1$- and $L^2$-Fourier transform  already appear in the present theory of quantum groups. In fact, they are implicitly used to define duals of quantum groups. Let us comment on this.

First of all, the $L^2$-Fourier transform is implicitly used in the construction of the Pontrjagin dual of a (von Neumann algebraic) quantum group. For the classical case of a locally compact abelian group $G$, let $(L^\infty(G), \Delta_G)$ be the usual quantum group associated with it. Its dual is given by $(\mathcal{L}(G), \hat{\Delta}_G)$, where $\mathcal{L}(G)$ is the group von Neumann algebra of $G$. This structure is spatially isomorphic to $(L^\infty(\hat{G}), \Delta_{\hat{G}})$ be means of the $L^2$-Fourier transform. That is $L^\infty(\hat{G}) = \mathcal{F}_2 \mathcal{L}(G)\mathcal{F}_2^{-1}$ and similarly the coproduct and other concepts translate. 

Secondly, in a paper by Van Daele \cite{DaeleFourier} he explains how the Fourier transform should be interpreted on the algebraic level of quantum groups. In his concluding remarks he suggests to study this transform in the operator algebraic framework. Here, this investigation is carried out. We take Van Daele's definition, which agrees with the classical transform (\ref{EqnIntroFT}), as a starting point for defining a $L^2$-Fourier transform in the operator algebraic framework.
 
Finally, an operator algebraic interpretation of the Fourier transform can be found in \cite{KahngFourier}. The main ideas for our $L^2$-Fourier transform first appear  here. However, the suggested Fourier transform \cite[Definition 3]{KahngFourier} is well-defined only if the Haar weights of a quantum group are states, i.e. if the quantum group is compact. In the more general situation, one has to give a more careful definition, which we work out in Section \ref{SectFT}.

\vspace{0.3cm}

The present paper is related to a
collection of papers studying module structures of $L^p$-spaces associated to the Fourier algebra of locally compact groups \cite{Daws}, \cite{DawsRunde}, \cite{ForLeeSam}. These papers are based on the theory of non-commutative
$L^p$-spaces associated to arbitrary (not necessarily semi-finite)
von Neumann algebras, which we  recall below.

 When dealing with these spaces, we are confronted with the following obstruction. For classical $L^p$-spaces associated with a measure space $X$, there is a clear understanding of the the intersections of $L^p$-spaces by means of disjunction of sets. So $L^p(X) \cap L^{p'}(X)$ gives the intersection of $L^p(X)$ and  $L^{p'}(X)$. For non-commutative $L^p$-spaces it is more difficult to find the intersection of two such spaces. In fact, there is a choice which determines the intersection and which depends on a {\it complex interpolation parameter} $z \in \mathbb{C}$. In \cite{ForLeeSam} the parameter $z = -1/2$ is used, whereas \cite{Daws} focuses on the case $z = 0$ in order to define module actions. In the final remarks of \cite{DawsRunde}, it is questioned which parameter would fit best for quantum groups.

 One of the results of the present paper is that to define a $L^p$-Fourier transform, one is obliged to choose the parameter $z = -1/2$. We also determine intersections of the $L^1$- and $L^2$-space and of the $L^2$- and $L^\infty$-space associated with a von Neumann algebra for this parameter, which are  natural spaces.

\subsection*{Structure of the paper}

  In Section \ref{SectLp} we recall the definition of non-commutative $L^p$-spaces and introduce the complex interpolation parameter $z \in \mathbb{C}$. We specialize the theory for  $z = -1/2$ and introduce short hand notation. In Sections \ref{SectIntersections} - \ref{SectApplications}, we only work with $L^p$-spaces with respect to this parameter. The justification for this specialization is given in the final chapter.

As indicated, the study of the  intersections of $L^p$-spaces becomes more intricate in the non-commutative setting. 
In Section \ref{SectIntersections} we determine the intersections of  $L^1$- and $L^2$-space and of the $L^2$- and $L^\infty$-space associated with a von Neumann algebra. These intersections turn out to be well-known spaces in the theory of quantum groups. This gives a confirmation that our choice for the interpolation parameter made at the beginning is a natural one. Moreover,   it gives the necessary ammunition to apply the re-iteration theorem. We warn the reader that the contents of Section \ref{SectIntersections} are relatively technical and if one is more interested in Fourier theory on quantum groups, one can skip Section \ref{SectIntersections} at  first reading.

Section \ref{SectLcqg} recalls the definition of a locally compact quantum group as given by Kustermans and Vaes. We give the von Neumann algebraic definition. 

In Section \ref{SectFT} we define the $L^p$-Fourier transform. We start with the $L^1$- and $L^2$-theory and then obtain the $L^p$-Fourier transform ($1 \leq p \leq 2$) through the complex interpolation method, a method similar to the Riesz-Thorin theorem mentioned in the introduction. 

In Section \ref{SectApplications} we define a convolution product in the $L^p$-setting and show that the Fourier transform turns the convolution product into a product.  

Finally, in Section \ref{SectSUTwo}, we prove that the interpolation parameter used in Sections \ref{SectIntersections} - \ref{SectApplications} is distinguished. That is, we prove that given the $L^2$-Fourier transform, there is only one choice for the interpolation parameter that allows an $L^p$-Fourier transform. This justifies our choice for this  parameter made in the beginning.

 \subsection*{Notations and conventions}

Throughout this paper, let $M$ be a von Neumann algebra and $\varphi$ a normal, semi-finite, faithful weight on $M$. We adopt the standard notations from \cite{TakII}. So $\mathfrak{n}_\varphi = \{ x
\in M \mid \varphi(x^\ast x) < \infty \}, \mathfrak{m}_\varphi =
\nphi^\ast \nphi$. We denote $\nabla, J, \sigma$ for the modular
operator, modular conjugation and modular automorphism group associated with $\varphi$.  We denote $\mathcal{T}_\varphi$ for the Tomita
algebra defined by
\[
 \mathcal{T}_\varphi = \left\{ x \in M \mid x \textrm{ is analytic w.r.t. } \sigma   \textrm{ and } \forall z \in \mathbb{C}: \sigma_z (x) \in \mathfrak{n}_\varphi \cap \mathfrak{n}_\varphi^\ast\right\}.
\]
Let $(\cH, \pi, \Lambda)$ be the GNS-representation of $M$ with respect to $\varphi$. Note that $M$ can be considered as acting on $\cH$ and therefore we omit the map $\pi$ if possible. For $x \in \mphi$ and for $a \in M$ analytic with respect to $\sigma$, we have $a x \in \mphi$ and $xa \in \mphi$ and
\[
\varphi(ax) = \varphi(x \sigma_{-i}a).
\]

%For a vector $\Lambda(x)$ in the full left Hilbert algebra $\Lambda(\nphi \cap \nphi^\ast)$ we denote $\pi_l(\Lambda(x))$ for the the bounded operator defined by $\pi_l(\Lambda(x)) \Lambda(y) = \Lambda(xy), y \in \nphi \cap \nphi^\ast$. If a vector $\Lambda(x)$ in the left Hilbert algebra $\Lambda(\nphi \cap \nphi^\ast)$ is right bounded, we denote $\pi_r(\Lambda(x))$ for the bounded operator defined by $\pi_r(\Lambda(x)) \Lambda(y) = \Lambda(yx), y \in \nphi \cap \nphi^\ast$.

 For a subset $A \subseteq M$, we denote $A^+$ for the positive elements in $A$. Similarly, $M_\ast^+$ denotes the space of positive normal functionals on $M$. Let $\omega \in M_\ast$. We donote $\overline{\omega} \in M_\ast$ for the functional defined by $\overline{\omega}(x) = \overline{\omega(x^\ast)}, x\in M$. For $y \in M$, we denote $y\omega$ and $\omega y$ for the normal functionals defined respectively by $(y \omega)(x) = \omega(xy)$ and $(\omega y)(x) = \omega(yx)$ with $x \in M$. Inner products on a Hilbert spaces are linear in the first entry and anti-linear in the second. Suppose that $M$ acts on a Hilbert space $\cH$ and let $\xi, \eta \in \cH$. We denote $\omega_{\xi, \eta}$ for the normal functional defined by $\omega_{\xi, \eta}(x) = \langle x \xi, \eta \rangle$.  The character $\iota$ will always stand for the identity homomorphism. If $x$ is a preclosed operator, we use $[x]$ for its closure.

\section{Non-commutative $L^p$-spaces}\label{SectLp} 

To any von Neumann algebra $M$, there is a way to associate a non-commutative $L^p$-space to  it. In fact there are many ways to do this. If $M$ is semi-finite, i.e. it admits a normal, semi-finite, faithful trace $\tau$, then one can define $L^p(M)$ as the space of closed, densely defined operators $x$ affiliated with $M$ for which if $\vert x \vert = \int_{[0,\infty)} \lambda dE_\lambda$ is the spectral decomposition of $\vert x \vert$, then
\[
\Vert x \Vert_p := \left( \sup_{n \in \mathbb{N}}  \tau \left( \int_{[0,n]} \lambda^p dE_\lambda \right) \right)^{1/p}< \infty.
\] 
If $M$ is abelian, we recover the classical spaces $L^p(X)$ for a certain measure space $X$.

Since the introduction of Tomita-Takesaki theory, $L^p$-spaces have been defined for von Neumann algebras that are not semi-finite. 
Definitions of non-commutative $L^p$-spaces have been given by Haagerup \cite{HaaLps}, \cite{TerpI},  Hilsum \cite{Hilsum}, Terp \cite{TerpII} and Izumi \cite{Izumi}. The definitions can be shown te be equivalent. That is, the $L^p$-spaces obtained by the various definitions are isometrically isomorphic  Banach spaces. 
For a good introduction to this theory, we refer to \cite{TerpI}, where a comparison of  Haagerup's definition and Hilsum's definition is made. 

\vspace{0.3cm}

Here we mainly use Izumi's definition \cite{Izumi} which is abstract in nature. He defines $L^p$-spaces associated with $M$ by means of the complex interpolation method; a method that admits a property that is reminiscent of the Riesz-Thorin theorem, see the introduction. It is for this reason that Izumi's definition is the most suitable context to work in.

 Drawback of this context is that the more concrete approach of the other definitions is absent. Whenever it feels appropriate we comment on this.

\subsection{The complex interpolation method}\label{SubSectInt}

We recall the complex interpolation method as explained in \cite[Section 4.1]{BerghLof}.

\begin{dfn}
Let $E_0, E_1$ be Banach spaces. The couple $(E_0, E_1)$ is called
a compatible couple (of Banach spaces) if $E_0$ and $E_1$ are continuously
embedded into a Banach space $E$.  
\end{dfn}
Note that we suppress $E$ in the notation $(E_0, E_1)$. We can consider the spaces $E_0 \cap E_1$ and $E_0 + E_1$ interpreted within $E$ and equip them with norms
\[
\begin{array}{ll}
\Vert x \Vert_{E_0 \cap E_1} = \max \{ \Vert x \Vert_{E_0}, \Vert x \Vert_{E_1} \}, & x \in E_0 \cap E_1, \\
\Vert x \Vert_{E_0+E_1} = \inf \{ \Vert x_0 \Vert_{E_0} +  \Vert x_1 \Vert_{E_1}  \mid x_0 + x_1 = x \},&  x \in E_0 + E_1,
\end{array}
\]
which make them Banach spaces. In that case we can consider $E_0$ and $E_1$ as subspaces of $E_0 + E_1$.  

\begin{dfn}
A morphism between compatible couples $(E_0, E_1)$ and $(F_0, F_1)$ is a bounded map $T: E_0 + E_1 \rightarrow F_0 + F_1$ such that for any $j \in \{0,1\}$,  $T(E_j) \subseteq F_j$ and the restriction $T: E_j \rightarrow F_j$ is bounded. 
\end{dfn}

\begin{rmk}\label{RmkCptMor}
Let $(E_0, E_1)$ and $(F_0, F_1)$ be compatible couples. If $T_0: E_0 \rightarrow F_0, T_1: E_1 \rightarrow F_1$ are bounded maps such that $T_0$ and $T_1$ agree on $E_0 \cap E_1$, then we call $T_0$ and $T_1$ compatible morphisms. In this case, there is a unique bounded map $T: E_0 + E_1 \rightarrow F_0 + F_1$. This gives a way to find morphisms of compatible couples. 
\end{rmk}

Now we describe the complex interpolation method. Let $(E_0, E_1)$ be a compatible couple. 
  Let $\sS = \{ z \in \mathbb{C} \mid 0 \leq {\rm Re}(z)
\leq 1 \}$ and let $\sS^\circ$ denote its interior. Let
$\mathcal{G}(E_0, E_1)$ be the set of functions $f: \sS \rightarrow
E_0 + E_1$ such that
\begin{itemize}
 \item[(1)] $f$ is bounded and continuous on $\sS$ and analytic on $\sS^\circ$;
\item[(2)] For $t \in \mathbb{R}, j\in \{0, 1\}$, $f(it + j) \in E_j$ and $t \mapsto f(it+j)$ is continuous and bounded with respect to the norm on $E_j$;
 \item[(3)] For $j\in \{0, 1\}$, $\Vert f(it+j) \Vert_{E_j} \rightarrow 0$ as $t \rightarrow \infty$.
\end{itemize}
Note that at this point our notation is different from
\cite{BerghLof} and \cite{Izumi}, where $\mathcal{G}$ is denoted by
$\mathcal{F}$, which we reserve for the Fourier transform. For $f \in \mathcal{G}(E_0, E_1)$, we define a
norm 
\[
\vertt f \vertt = \max \{ \sup \Vert f(it) \Vert_{E_0}, \sup
\Vert f(it+1) \Vert_{E_1} \}.
\]
 Let $\theta \in [0,1]$. We define
$(E_0, E_1)_{[\theta]} \subseteq E$ to be the space $\{
f(\theta) \mid f \in  \mathcal{G}(E_0, E_1) \}$ with norm
\[
\Vert x \Vert_{[\theta]} = \inf \{ \vertt f \vertt \mid f(\theta)
= x, f \in \mathcal{G}(E_0, E_1) \}.
\]
 With this norm, $(E_0, E_1)_{[\theta]}$ is a Banach space \cite[Theorem 4.1.2]{BerghLof}. 

\begin{dfn} \label{DfnMorphCpt} 
The assignment from compatible couples of Banach spaces to Banach spaces which is given by $C_\theta: (E_0, E_1)  \rightarrow (E_0, E_1)_{[\theta]}$ is called the complex interpolation method (at parameter $\theta \in [0,1]$). $(E_0, E_1)_{[\theta]}$ is called a complex interpolation space. 
\end{dfn}

The following Riesz-Thorin-like theorem plays a central role in the present paper. It gives the functorial property of the complex interpolation method. 

\begin{thm}[Theorem 4.1.2 of \cite{BerghLof}]\label{ThmCplxInterpolationBound}
Let $\theta \in [0,1]$. Let $T$ be a morphism between compatible couples $(E_0, E_1)$ and $(F_0, F_1)$. Then, it restricts to a bounded linear map $T: (E_0, E_1)_{[\theta]} \rightarrow (F_0, F_1)_{[\theta]}$. The norm is bounded by $\Vert T \Vert \leq \Vert T: E_0 \rightarrow F_0 \Vert^{1-\theta} \Vert T: E_1 \rightarrow F_1 \Vert^{\theta} $.
\end{thm}

If we let $C_\theta$ of Definition \ref{DfnMorphCpt} act on the morphisms $T: (E_0, E_1) \rightarrow (F_0, F_1)$ of compatible couples by assigning its restriction $T: (E_0, E_1)_{[\theta]} \rightarrow (F_0, F_1)_{[\theta]}$ to it, we see that $C_\theta$ is a functor. 

\begin{rmk}
Using the notation of Remark \ref{RmkCptMor}, the compatible    morphisms $T_0, T_1$ give rise to a morphism $C_\theta(T): (E_0, E_1)_{[\theta]} \rightarrow (F_0, F_1)_{[\theta]}$ on the interpolation spaces with norm $\Vert C_\theta(T) \Vert \leq  \Vert T_0 \Vert^{1-\theta} \Vert T_1 \Vert^{\theta} $.
\end{rmk}

We will need the following useful fact, see \cite[Theorem 4.2.2]{BerghLof}.

\begin{lem}\label{LemDenseInt}
Let $(E_0, E_1)$ be a compatible couple and $ \theta \in [0,1]$. $E_0 \cap E_1$ is dense in $(E_0, E_1)_{[\theta]}$. 
\end{lem}

\subsection{Izumi's $L^p$-spaces}
 In \cite{TerpII} Terp shows that the $L^p$-spaces as introduced by Hilsum can be obtained by applying the complex interpolation method to a specific compatible couple $(M, M_\ast)$, see \cite[Theorem 36]{TerpII}. Izumi \cite{Izumi} realized that there is more than one way to turn $(M, M_\ast)$ into a compatible couple in order to obtain the $L^p$-spaces through interpolation. His idea is to {\it define} non-commutative $L^p$-spaces as complex interpolation spaces of certain compatible structures. It is this definition which we recall here.  

Here, we present the general picture. However, in the larger part of the present paper, we only work with the complex interpolation parameters $z = -1/2$ and $z = 1/2$ (we introduce the parameter in a minute). We will specialize the theory for these parameters in Sections \ref{SectHalf} and \ref{SectGNS} and introduce short hand notation there. 
The more general theory is used in Section \ref{SectSUTwo}, where we prove that there is in principle only one interpolation parameter that allows a $L^p$-Fourier transform, namely $z = -1/2$.   

Fix a von Neumann algebra $M$ with normal, semi-finite, faithful weight $\varphi$. The following construction of $L^p$-spaces can be found in \cite{Izumi}.

\begin{dfn}\label{DfnL}
 For $z \in \mathbb{C}$, we put
\[
 L_{(z)} = \left\{ x \in M \mid
\exists \varphi^{(z)}_x \in M_\ast \textrm{ s.t. } \forall a,b \in \mathcal{T}_\varphi: 
\varphi^{(z)}_x(a^\ast b) = \langle x J \nabla^{\bar{z}} \Lambda(a) \mid J \nabla^{-z} \Lambda(b) \rangle
\right\}.
\]
The number $z \in \mathbb{C}$ will be called the complex interpolation parameter. 
\end{dfn}

\begin{rmk}\label{RmkState}
We will mainly be dealing with the cases $z = -1/2$ and $z = 1/2$. Note that if $\varphi$ is a state, then for any $x \in M$, we see that for $a,b \in \mathcal{T}_\varphi$,
\[\langle x J \nabla^{-1/2} \Lambda(a) \mid J \nabla^{1/2} \Lambda(b) \rangle  =
\langle x J \nabla^{1/2} \Lambda(\sigma_{i}(a)) \mid J \nabla^{1/2} \Lambda(b) \rangle = \varphi(b  x \sigma_{-i}(a^\ast)) =  \varphi(a^\ast b x  ),
\]
and hence $L_{(-1/2)} = M$ and $\varphi^{(-1/2)}_x  = x \varphi$. Similarly, $L_{(1/2)} = M$ and  $\varphi^{(1/2)}_x  =   \varphi x$.
\end{rmk}

The following proposition  implies that there are   plenty of elements contained in $L_{(z)}$.

\begin{prop}[Propostion 2.3 of \cite{Izumi}]
$\mathcal{T}_\varphi^2 = \left\{ ab \mid a,b \in \mathcal{T}_\varphi \right\}$ is contained in $L_{(z)}$.
\end{prop}

We are now able to construct Izumi's $L^p$-spaces using the complex interpolation method. First, we define a compatible couple.   For $x\in L_{(z)}$, we define a norm:
\[
\Vert x \Vert_{L_{(z)}} = \max \{ \Vert x \Vert, \Vert \varphi^{(z)}_x \Vert \}.
\] 
We define norm-decreasing injections: 
\[
i_{(z)}^\infty:  \: L_{(z)} \rightarrow M: x \mapsto x; \qquad 
i_{(z)}^1:  \: L_{(z)} \rightarrow M_\ast: x \mapsto \varphi_x^{(z)}.
\]
Using the duals of the maps, we obtain the following diagram. Note that $(i^\infty_{(-z)})^\ast: M^\ast \rightarrow L_{(-z)}^\ast$ is restricted to $M_\ast$.
 
\begin{equation}
%    \xymatrix{
% & M_\ast\ar@{^{(}->}[dr]^{\nu_1^{(z)}} & \\
%L_{(z)} \ar@{^{(}->}[ur]^{\mu_1^{(z)}}\ar@{^{(}->}[dr]_{\mu_\infty^{(z)}}&& %L_{(-z)}^\ast; \\
%& M \ar@{^{(}->}[ur]_{\nu_\infty^{(z)}}&} 
%\quad
 \xymatrix{
 & M_\ast\ar@{^{(}->}[dr]^{(i^\infty_{(-z)})^\ast} & \\
   L_{(z)} \ar@{^{(}->}[ur]^{i^1_{(z)}}\ar@{^{(}->}[dr]_{i^\infty_{(z)}}\ar@{^{(}->}[r]^{i^p_{(z)}}& L^p_{(z)}(M)\ar@{^{(}->}[r]& L_{(-z)}^\ast. \\
& M \ar@{^{(}->}[ur]_{(i^1_{(-z)})^\ast}&}    \label{EqnLpIzu}
\end{equation}
Now, \cite[Theorem 2.5]{Izumi} yields that the outer rectangle of (\ref{EqnLpIzu}) commutes. This turns $(M,M_\ast)$ into a compatible couple of Banach spaces.

\begin{dfn}
For $p \in (1, \infty)$, we define $L^p_{(z)}(M)$ to be the complex interpolation space $(M,M_\ast)_{[1/p]}$. We set $L^1_{(z)}(M) = M_\ast$ and  $L^\infty_{(z)}(M) = M$ .
\end{dfn}
By Lemma \ref{LemDenseInt}, $L_{(z)}$ can be embedded in $L^p_{(z)}(M)$. This map is denoted by $\iota^p_{(z)}$.
Note that by definition of the complex interplation method $L^p_{(z)}(M )$ is a linear subspace of $L_{(-z)}^\ast$.

\begin{nt}\label{NotSubsp}
The map $i_{(z)}^\infty:  \: L_{(z)} \rightarrow M $ is basically the inclusion of a subspace. Therefore, it is   convenient to omit  the map $i_{(z)}^\infty$ in our notation if the norms of the spaces do not play a role in the statement. Similarly, we do not introduce notation for the inclusion of $L^p_{(z)}(M)$ in $L_{(-z)}^\ast$, where $p \in (1, \infty)$.    
\end{nt}

A priori one could think that $L^p_{(z)}(M)$ and  $L^p_{(z')}(M)$ with $z \not = z'$, are different as Banach spaces. However, Izumi proves that they are isometrically isomorphic. Terp \cite{TerpII} considers the case $z = 0$. The main result of \cite{TerpII} is that  $L^p_{(0)}(M)$ is isometrically isomorphic to the $L^p$-spaces by   Hilsum \cite{Hilsum}. We will come back to this Section \ref{SectGNS}.

\begin{thm}[Theorem 3.8 of \cite{Izumi}] \label{ThmIsomLpSpaces}
For $z, z' \in \mathbb{C}$, there is an isometric isomorphism
\[
 U_{p, (z', z)}: L^p_{(z)}(M) \rightarrow  L^p_{(z')}(M), \qquad p \in (1, \infty),
\]
such that for $a \in \mathcal{T}_\varphi^2$,
\begin{equation}\label{EqnLpIsomorphism}
  U_{p, (z', z)}(i^{p}_{(z)} (a)) = i^{p}_{(z')}(\sigma_{i \frac{r'-r}{p}-(s'-s)}(a)),
\end{equation}
where $z = r+is$ and $z' = r'+is'$, $r, r', s, s' \in \mathbb{R}$.  
\end{thm}

We emphasize, that although the $L^p$-spaces appearing in (\ref{EqnLpIzu}) are isomorphic for different complex interpolation parameters, the intersections defined by this figure may be different. In any case, by \cite[Corollary 2.13]{Izumi},
\begin{equation}\label{EqnIzumiIntersection}
(i_{(-z)}^1)^\ast(L_{(z)} )= (i_{(-z)}^1)^\ast (M) \cap (i_{(-z)}^\infty)^\ast( M_\ast),
\end{equation}
i.e. if one consideres $L_{(z)}$, $M$, $M_\ast$ as subspaces of $L_{(-z)}^\ast$, then  $L_{(z)} = M \cap M_\ast$.

\subsection{Specializations for the complex interpolation parameters} \label{SectHalf} 

In the present paper we will mainly work with the parameter $z = -1/2$. In order to study these spaces also the parameter $z = 1/2$ will play a role. In  this section, we specialize the theory for these parameters. 
 The following proposition shows that $L_{(-1/2)}$ and $L_{(1/2)}$ can be described by a condition that is in generally more easy to check. If $\varphi$ is a state it reduces to Remark \ref{RmkState}.
\begin{prop} \label{PropAltDescr}We have the following alternative descriptions:
\begin{enumerate} 
 \item\label{ItemAltDescr} Let $L=\{ x \in \nphi \mid \exists \: _x \varphi \in M_\ast\: {\rm  s.t. } \: \forall y \in \nphi:  \: _x \varphi (y^\ast) = \varphi(y^\ast x) \}.$ Then, $L =  L_{(-1/2)}$. 
\item\label{ItemAltDescrII}  Let $R = \{ x \in \nphi^\ast \mid \exists \varphi_x \in M_\ast\: {\rm  s.t. }\: \forall y \in \nphi: \: \varphi_x (y) = \varphi(xy) \}.$ Then, $R =  L_{(1/2)}$.
\end{enumerate}
\end{prop}
\begin{proof}
We only give the proof of (\ref{ItemAltDescr}), since (\ref{ItemAltDescrII}) can be proved similarly. We first prove $\subseteq$. For $ x\in L$, $a, b \in \mathcal{T}_\varphi$,
\[
 \begin{split}
	&  _x \varphi (a^\ast b) = \varphi(a^\ast bx) = \varphi( bx \sigma_{-i}(a^\ast)) = \langle x \Lambda(\sigma_{-i}(a^\ast)), \Lambda(b^\ast) \rangle \\ = & \langle x \nabla J \nabla^{1/2} \Lambda(a), J \nabla^{1/2}\Lambda(b) \rangle =
   \langle x J \nabla^{-1/2}  \Lambda(a),  J \nabla^{1/2} \Lambda(b) \rangle.
 \end{split}
\]
Hence $x \in L_{(-1/2)}$ and $_x \varphi = \varphi^{(-1/2)}_x$.

%If we prove that
%\begin{equation}\label{EqnAlternativeIntersection}
% L := \nu_\infty\left(\{ x \in \nphi \mid \exists \omega \in M_\ast\: {\rm  s.t. } \forall y \in \nphi \: \omega(y^\ast) = \langle \Lambda(x), %\Lambda(y) \rangle \}\right) = \nu_\infty^{(-1/2)}(M) \cap \nu_1^{(-1/2)}{M_\ast},
%\end{equation}
%the proposition follows since $L_{(-1/2)} = \nu_\infty^{(-1/2)}(M) \cap \nu_1^{(-1/2)}{M_\ast}$, see \cite[Corollary 2.13]{Izumi}.

To prove $\supseteq$, we first prove that $M \mathcal{T}_\varphi^2 \subseteq L_{(-1/2)}$. Indeed, let $x \in M$ and let $c,d \in \mathcal{T}_\varphi$. The functional $M \ni y \mapsto \varphi(\sigma_i(d) yx c)$ is normal. Furthermore, for $a,b \in \mathcal{T}_\varphi$,
\[
\langle xcd J \nabla^{-1/2} \Lambda(a), J \nabla^{1/2} \Lambda(b) \rangle = 
\langle \Lambda(xcd\sigma_{-i}(a^\ast)), \Lambda(b^\ast) \rangle =
\varphi(bxcd \sigma_{-i}(a^\ast)) = \varphi(\sigma_i(d) a^\ast bxc ).
\]
Hence, $xcd \in L_{(-1/2)}$. 

Next, we prove that $L_{(-1/2)} \subseteq \nphi$. Take $x \in L_{(-1/2)}$ and let $(e_j)_{j \in J}$ be a bounded net in $\mathcal{T}_\varphi$ such that $\sigma_i(e_j)$ is bounded and such that $e_j \rightarrow 1$ $\sigma$-weakly, see \cite[Lemma 9]{TerpII}. Then, $xe_j \rightarrow x$ $\sigma$-weakly. Furthermore,
\begin{equation}\label{EqnNormEstimate}
\Vert \Lambda(xe_j) \Vert^2 = \varphi(e_j^\ast x^\ast x e_j) = \varphi^{(-1/2)}_{x e_j \sigma_{-i}(e_j^\ast)}(x^\ast) \leq \Vert \varphi^{(-1/2)}_{x e_j \sigma_{-i}(e_j^\ast)} \Vert \Vert x \Vert,
\end{equation}
where the second equality is due to the previous paragraph. By \cite[Proposition 2.6]{Izumi}, 
\begin{equation}\label{EqnNormEstimateII}
\varphi^{(-1/2)}_{x e_j \sigma_{-i}(e_j^\ast)} = \varphi^{(-1/2)}_x   \sigma_i(e_j) e_j^\ast,
\end{equation}
 where for $\omega \in M_\ast, y \in M$, $\omega  y$ is the normal functional defined by $(\omega   y)(a) = \omega(ya), a \in M$. From (\ref{EqnNormEstimate}) and (\ref{EqnNormEstimateII}) it follows that $(\Lambda(x e_j))_{j\in J}$ is a bounded net. Furthermore, for $a, b \in \mathcal{T}_\varphi$,
\[
\langle \Lambda(xe_j), \Lambda(ab) \rangle = \varphi(b^\ast a^\ast x e_j  ) =
\varphi( a^\ast x e_j \sigma_{-i}(b^\ast) ) \rightarrow \varphi(a^\ast x \sigma_{-i}(b^\ast)).
\]
Since $(\Lambda(xe_j))_{j\in J}$ is bounded, this proves that $(\Lambda(xe_j))_{j\in J}$ is weakly convergent. Since $\Lambda$ is $\sigma$-weak/weak closed, this implies that $x \in \Dom(\Lambda) = \nphi$. So $L_{(-1/2)} \subseteq \nphi$.

To finish the proof, let again $x \in L_{(-1/2)}$ and let $a, b \in \mathcal{T}_\varphi$. We prove that $\varphi_x^{(-1/2)}((ab)^\ast) = \langle \Lambda(x), \Lambda(ab) \rangle$. The proposition then follows since by Lemma \ref{LemCoreT}, $\mathcal{T}_\varphi^2$ is a $\sigma$-weak/weak-core for $\Lambda$.  
The proposition follows from:
\[
 \langle \Lambda(x), \Lambda(ab) \rangle = \varphi(b^\ast a^\ast x) = \varphi(a^\ast x \sigma_{-i}(b^\ast)) = 
 \langle x J \nabla^{-1/2} \Lambda(b), J \nabla^{1/2} \Lambda(a^\ast)  \rangle  = 
  \varphi_x^{(-1/2)}(b^\ast a^\ast). 
\]
\begin{comment}
To finish the proof, let again $x \in L_{(-1/2)}$ and let $y \in \nphi \cap \nphi^\ast$. We prove that $\varphi_x^{(-1/2)}(y^\ast) = \langle \Lambda(x), \Lambda(y) \rangle$. The proposition then follows by the easily established fact that $\nphi \cap \nphi^\ast$ is a $\sigma$-weak/weak-core for $\Lambda$.  Put 
\[
y_n = \frac{n}{\sqrt{\pi}} \int_{-\infty}^{\infty} e^{-(nt)^2} \sigma_t^{\varphi}(y) dt,
\]
where the integral is taken in the $\sigma$-strong-$\ast$ sense. By standard techniques (c.f. the proof of \cite[Lemma 9]{TerpII}), $y_n \in \mathcal{T}_\varphi$, $y_n$ is bounded and $y_n$ converges $\sigma$-weakly to $y$. Moreover, using the the fact that $\Lambda$ is  $\sigma$-strong-$\ast$/norm closed, we obtain
\[
\Lambda(y_n) =  \frac{n}{\sqrt{\pi}} \int_{-\infty}^{\infty} e^{-(nt)^2} \nabla^{it} \Lambda(y) dt \rightarrow \Lambda(y) \qquad {\rm weakly},
\]
where the integral is a Bochner integral, c.f. \cite[Chapter VI, Lemma 2.4]{TakII}.  Let $(e_j)_{j \in J}$ be a bounded net in $\mathcal{T}_\varphi$ converging $\sigma$-weakly to 1. The proposition follows from:
\[
\begin{split}
& \langle \Lambda(x), \Lambda(y) \rangle = 
\lim_{j \in J} \lim_n \langle \Lambda(x), e_j \Lambda(y_n) \rangle \\ = &
\lim_{j \in J} \lim_n \langle x J \nabla^{-1/2} \Lambda(y_n), J \nabla^{1/2} \Lambda(e_j^\ast)  \rangle  = 
\lim_{j \in J} \lim_n \varphi_x^{(-1/2)}(y_n^\ast e_j) = \varphi_x^{(-1/2)}(y^\ast).
\end{split}
\]
\end{comment}
\end{proof}

In particular, it follows from Proposition \ref{PropAltDescr} that for $y \in \nphi$,
\begin{equation}
\begin{array}{rll}
_x \varphi (y^\ast) = & \varphi(y^\ast x),\qquad &   x \in L, \\
\varphi_x(y) = & \varphi(xy), & x \in R. 
\end{array}\label{EqnVarphiLeftRight}
\end{equation}
We emphasize that one has to be careful that (\ref{EqnVarphiLeftRight}) does not make sense for every $x,y \in M$. Also, (\ref{EqnVarphiLeftRight})justifies why (\ref{EqnLpIzu}) is also called the left injection for $z = -1/2$ and the  right injection for $z = 1/2$.

 Part of the next Corollary is already proved in \cite{Izumi}. Using the alternative descriptions of Proposition \ref{PropAltDescr}, it is easy to prove the remaining statements.

\begin{cor}\label{CorInclusions}
  We have inclusions $M \mathcal{T}_\varphi^2 \subseteq L$, $\mathcal{T}_\varphi^2 M \subseteq R$, $\mathcal{T}_\varphi^2 \subseteq L \cap R$, $L \mathcal{T}_\varphi \subseteq L$, $\mathcal{T}_\varphi R \subseteq R$, $M L \subseteq L$ and $R M \subseteq R$. Moreover, $R = \{ x^\ast \mid x  \in L \}$ and for $x \in L$, $\varphi_{x^\ast} = \overline{_x \varphi}$. 
\end{cor}
\begin{proof} 
The first inclusion has already been proved in the proof of Proposition \ref{PropAltDescr}. Here we have proved that for $x \in M, a,b \in \mathcal{T}_\varphi$, $_{xab} \varphi (z) = \varphi(\sigma_{i} (b) \: z \: xa  ), z \in M$. Similarly, one can prove that for $x, z \in M, a, b \in \mathcal{T}_\varphi, y_l \in L , y_r \in R$,
\[
\begin{array}{lll}
\varphi _{ab x}(z) = \varphi(bx \: z \: \sigma_{-i}(a));&
\varphi _{ab}(z) = \varphi(b \: z \: \sigma_{-i}(a));&
_{ab} \varphi (z) = \varphi( \sigma_{i}(b) \: z \: a); \\
_{y_l a} \varphi (z) =  _{y_l} \varphi(\sigma_{i}(a) z); &
\varphi _{ a y_r}(z) = \varphi^{(1/2)}_{y_r}( z \sigma_{-i}(a));&
_{x y_l} \varphi (z) = _{y_l} \varphi ( z x); \\
\varphi _{ y_r x}(z) = \varphi _{y_r}( x z );&
\varphi_{x^\ast}  = \overline{_{x} \varphi } .&
\end{array}
\]
\end{proof}

Since we are mainly dealing with complex interpolation parameter $z = -1/2$ and $z = 1/2$, it is more convenient to adapt our notation.
\begin{nt}\label{NotMMast}
We use the following short hand notations. For $p \in [1, \infty]$,
\[
\begin{array}{llll}
 L^p(M)_{{\rm left}} = L^p_{(-1/2)}(M),& L = L_{(-1/2)}, & l^p = i_{(-1/2)}^p,& _x \varphi  = \varphi_x^{(-1/2)} \: {\rm for } \: x \in L, \\
 L^p(M)_{{\rm right}} = L^p_{( 1/2)}(M),& R = L_{( 1/2)}, & r^p = i_{( 1/2)}^p, &   \varphi_x  = \varphi_x^{(1/2)} \: {\rm for } \: x \in R.
\end{array}
\]
Recall that by definition $M_\ast = L^1(M)_{{\rm left}}$ and   $M = L^\infty(M)_{{\rm left}}$. From now on, we consider $M_\ast$ and $M$ as subspaces of $R^\ast$ by means of the respective maps $r_\infty^\ast$ and $r_1^\ast$ and it is convenient to omit these maps in the notation. So the identifications of $M_\ast$ and $M$ in $R^\ast$ are given by the pairings:
\begin{eqnarray}
\langle \omega, y \rangle_{R^\ast, R} =& \omega(y),& \qquad \omega \in M_\ast, y \in R, \label{EqnPairingMast} \\
 \langle x, y \rangle_{R^\ast, R} =& \varphi_y(x),  & \qquad x \in M, y \in R. \label{EqnPairingM}
\end{eqnarray}
The norm on $L$ will be denoted by $\Vert \cdot \Vert_L$.
\end{nt}

\subsection{Comparison with  Hilsum's $L^p$-spaces}\label{SectGNS}

  Here, we recall the definition of non-commuta-tive $L^p$-spaces given in \cite{Hilsum}, see also \cite{TerpII}. We need these spaces for two reasons. 

First of all, many of the objects we introduce are constructed by means of Theorem \ref{ThmCplxInterpolationBound}. For that reason the structures are abstract in nature. The advantage of the  Hilsum approach is that it is much more concrete. Hence, also the objects defined in Section \ref{SectFT} have a more concrete meaning when  they are considered in the  Hilsum setting.  

Secondly, a non-commutative $L^2$-space associated with a von Neumann algebra $M$ with weight $\varphi$ can be identified with the GNS-space $\cH$ of the weight. In \cite[Theorem 23]{TerpII} this identification is given for Hilsum's definition. In \cite{Izumi}, Izumi does not explicitly keep track of an isomorphism between $L^2(M)_{{\rm left}}$ with $\cH$. Here we make this isomorphism explicit. This is useful for the $L^p$-Fourier transform. In particular, Corollary \ref{CorInverse} relies heavily on this identification. 

\vspace{0.3cm}

We refer to the original paper \cite{Hilsum} for  Hilsum's $L^p$-spaces. The following is also nicely summarized in \cite[Sections III and IV]{TerpI}.
Fix a normal, semi-finite, faithful weight $\phi$ on the commutant $M'$. Let $\sigma^\phi$ be its modular automorphism group.

\begin{dfn}
 A closed, densely defined operator $x$ on $\cH$ is called $\gamma$-homogeneous, with $\gamma \in \mathbb{R}$ if the following skew commutation relation holds
\[
xa \subseteq a \sigma_{i \gamma}^\phi (x), \qquad \textrm{ for all } \: a \in M' \textrm{ analytic w.r.t. } \sigma^\phi. 
\]
\end{dfn}

The following theorem requires the spatial derivate \cite{Con}, \cite{TakII}. We will not recall this construction, but rather cite its properties. For a good introduction, we refer to \cite[Section III]{TerpII}. The spatial derivative construction gives a passage between $M_\ast$ and the $(-1)$-homogeneous operators. The following theorem can be found under the given references in \cite{TerpI}. It can be derived from \cite[Theorem 13]{Con}.

\begin{thm}[Theorem 29, Definition 33 and Corollary 34 of \cite{TerpI}] \label{ThmSpatDer} 
Let $x$ be a closed, densely defined operator with polar decomposition $x = u \vert x \vert$. Let $p \in [1, \infty]$. The following are equivalent:
\begin{enumerate}
\item $x$ is $(-1/p)$-homogeneous;
\item $u \in M$ and $\vert x \vert^p$ is $(-1)$-homogeneous;
\item $u \in M$ and there is a normal, semi-finite weight $\psi$ on $M$ such that $\vert x \vert^p$ equals the spatial derivative $d\psi/d\phi$. 
\end{enumerate}
\end{thm}

\begin{dfn}
Let $p \in [1,\infty)$. The  Hilsum $L^p$-space $L^p(\phi)$ is defined as the space of closed, densely defined operators $x$ on the GNS-space $\cH$ of $\varphi$ such that if $x = u \vert x \vert$ is the polar decomposition, then $\vert x \vert^p$ is the spatial derivative of a positive $\omega \in M_\ast$ and $u \in M$. It carries the norm $ \Vert x \Vert_p = (\omega(1))^{1/p}$. We set $L^\infty(M) = M$.
\end{dfn}

In particular, every operator in $L^p(\phi)$ is closed, densely defined and $(-1/p)$-homogeneous. This includes $p = \infty$. By Theorem \ref{ThmSpatDer}, the spatial derivative gives an isometric isomorphism between $M_\ast$ and $L^1(\phi)$.

We introduce notation for the distinguished spatial derivative
\[
d = d\varphi/d\phi.
\]
  $d$ is a strictly positive, self-adjoint operator acting on the GNS-space $\cH$. 
 We need the fact that it implements the modular automorphism group of $\varphi$ and $\phi$, i.e.
\[
\sigma_t(x) = d^{it} x d^{-it}, \:\: x \in M \qquad \sigma^\phi_t(x) = d^{-it} y d^{it}, \:\: y\in M' .
\]
Using this, one can prove that $d$ is $(-1)$-homogeneous, see \cite[Lemma 22]{TerpII}. The operator $d$ forms a handy tool to find elements of $L^p(\phi)$.

\begin{lem}[Theorem 26 of \cite{TerpII}]\label{LemInc}
Let $p \in [2, \infty]$ and let $x \in \nphi$. Then, $x d^{1/p}$ is preclosed and its closure $[ x d^{1/p}]$ is in $L^p(\phi)$. Moreover, there is an isometric isomorphism $\mathcal{P}: \cH \rightarrow L^2(\phi)$ given by $[xd^{1/2}] \mapsto \Lambda(x)$.
\end{lem}

We use this result to prove the following.

\begin{prop}\label{PropL2Identifications}
Let $p \in [1, \infty]$. 
\begin{enumerate}
\item\label{ItemComparisonI} Let $a,b \in \mathcal{T}_\varphi$. Then, $ab d^{1/p}$ is preclosed and its closure $[ ab d^{1/p}]$ is in $L^p(\phi)$.
\item\label{ItemComparisonII} There is an isometric isomorphism $\Phi_p: L^p(\phi) \rightarrow L^p(M)_{{\rm left}}$ such that 
\[
\Phi_p: [ab d^{1/p}] \mapsto l^p(ab), \qquad  a,b \in \mathcal{T}_\varphi.
\] 
\item\label{ItemComparisonIII}  There is a unitary map $U_l: L^2(M)_{{\rm left}} \rightarrow \cH$ determined by 
\[
U_l: l^2(a) \mapsto \Lambda(a), \qquad a \in \mathcal{T}_\varphi^2.
\] 
\item\label{ItemComparisonIV} More general, there is a unitary map $U_{(z)}: L^2_{(z)}(M) \rightarrow \cH$ determined by 
\[
U_{(z)}: i^2_{(z)}(a) \mapsto \Lambda(\sigma_{-i(z/2+1/4)} (a) ), \qquad a \in \mathcal{T}_\varphi^2.
\] 
\end{enumerate}
\end{prop}
\begin{proof}
(\ref{ItemComparisonI}) First note that using \cite[Lemma 22]{TerpII} for the first inclusion, Lemma \ref{LemInc} for the third and \cite[Theorem 4 (3)]{Hilsum} for the last, 
\[
ab d^{1/p} \subseteq  d^{2/p} \sigma_{2i/p}(ab)d^{2/p} \subseteq d^{2/p} \sigma_{2i/p}(a) \cdot [\sigma_{2i/p}(b) d^{2/p}] \in L^{p/2}(\phi) \cdot  L^{p/2}(\phi) \subseteq L^{p}(\phi).
\]
Hence,  
\begin{equation}
(ab d^{1/p})^\ast \supseteq ( d^{2/p} \sigma_{2i/p}(a) \cdot [\sigma_{2i/p}(b) d^{2/p}])^\ast \in   L^{p}(\phi).
\end{equation}
So that $(ab d^{1/p})^\ast$ is densely defined. Hence $ab d^{1/p}$ is preclosed and  by (the proof of) \cite[Theorem 4 (1)]{Hilsum}, $(ab d^{1/p})^\ast = ( d^{2/p} \sigma_{2i/p}(a) \cdot [\sigma_{2i/p}(b) d^{2/p}])^\ast$, hence $[ab d^{1/p}] = d^{2/p} \sigma_{2i/p}(a) \cdot [\sigma_{2i/p}(b) d^{2/p}] $.

(\ref{ItemComparisonII}) It is argued in the remarks following \cite[Proposition 2.4]{Izumi} that (\ref{EqnLpIzu}) for $z = 0$ equals the compatible couple as considered in \cite{TerpII}. First note that by \cite[Eqn. (50)]{TerpII},
\[
[ab d^{1/p}]  = d^{2/p} \sigma_{2i/p}(a) \cdot [\sigma_{2i/p}(b) d^{2/p}]  = \mu_p(\sigma_{2i/p}(ab)),
\] 
where $\mu_p$ is the embedding of $L_{(0)}$ in $L^p(\phi)$, see \cite[Theorem 27]{TerpII}. The main result of \cite{TerpII} is that $L^p(\phi)$ is isometrically isomorphic to $L^p_{(0)}(M)$. The isomorphism is given by the map $\nu_p:   L^p(\phi) \rightarrow L^p_{(0)}(M)$ of \cite[Theorem 30]{TerpII}. Moreover, we see that $\nu_p \mu_p   = (i^1_{(0)})^{\ast} i^\infty_{(0)} $ by commutativity of \cite[Eqn. (55)]{TerpII}. In turn we have $(i^1_{(0)})^{\ast} i^\infty_{(0)}  = i^p_{(0)}  $ by commutativity of (\ref{EqnLpIzu}). Hence, we have an isometric isomorphism $L^p(\phi) \rightarrow  L^p_{(0)}(M)$ for which
\[
[ab d^{1/p}] \mapsto \nu_p([ab d^{1/p}]) =  \nu_p \mu_p(\sigma_{2i/p}(ab)) = i^p_{(0)}(\sigma_{2i/p}(ab)).
\]
We conclude the proof by applying the isometric isomorphism $U_{(-1/2, 0)}$ of Theorem \ref{ThmIsomLpSpaces}, so that we we get an isometric isomorphism $\Phi_p: L^p(\phi) \rightarrow L^p_{(-1/2)}(M) = L^p(M)_{{\rm left}}$, such that 
\[
\Phi_p: [ab d^{1/p}] \mapsto U_{(-1/2, 0)} i^p_{(0)}(\sigma_{2i/p}(ab)) = i^p_{(-1/2)} (ab) = l^p(ab), \qquad a,b \in \mathcal{T}_\varphi.
\]

(\ref{ItemComparisonIII}) This follows from (\ref{ItemComparisonII}) by applying  Lemma \ref{LemInc} and the fact that $\Lambda(\mathcal{T}_\varphi^2)$ is dense in $\cH$. So $U_l = \Phi_p^{-1} \mathcal{P}^{-1}$. (\ref{ItemComparisonIV}) $U_{(z)} = U_l U_{2, (-1/2, z)}$. 
\end{proof}

Recall that $L^2(M)_{{\rm left}}$ is by definition  a subspace of $R^\ast$. Therefore, we can pair elements of $L^2(M)_{{\rm left}}$ with elements of $R$.

\begin{prop}\label{PropIzumiComp}
For $\xi \in \cH$, $y \in R$, 
\[
\langle U_l^\ast \xi , y \rangle_{R^\ast, R}  = \langle \xi, \Lambda(y^\ast) \rangle.
\]
\end{prop}
\begin{proof}
 First assume that $\xi = \Lambda(x) =  U_l l_2 (x), x \in L$. Using the commutativity of (\ref{EqnLpIzu}) in the second equality,
\[  
\begin{split}
& \langle  U_l^\ast \xi  , y \rangle_{R^\ast, R} =    \langle  l_2(x)  , y \rangle_{R^\ast, R}   
 =   \langle   l_1(x)  , y \rangle_{ R^\ast, R} \\ = \!\!\!\!\!\!\!\!^{(\ref{EqnPairingMast})}&    (_x \varphi)(y) = \!\!\!\!\!\!\!\!^{(\ref{EqnVarphiLeftRight})} \varphi(yx) = \langle \Lambda(x), \Lambda(y^\ast) \rangle  = \langle \xi , \Lambda(y^\ast) \rangle .  
\end{split}
\]
 The proposition follows by the fact that $\Lambda(\mathcal{T}_\varphi^2) \subseteq \Lambda(L)$ is dense in $\cH$.
\end{proof}

\begin{nt}\label{NotGNS}
From now on, we will identify $\cH$ and $L^2(M)_{{\rm left}}$ and consider it as a subspace of $R^\ast$. The identification is given via the unitary $U_l$. Under this identification the map $l_2$ becomes the GNS-map $\Lambda$, see Proposition \ref{PropL2Identifications}. By Proposition \ref{PropIzumiComp} we see that $\cH$ is identified as a subspace of $R^\ast$ by means of the pairing
\begin{equation}
\langle  \xi , y \rangle_{R^\ast, R}  = \langle \xi, \Lambda(y^\ast) \rangle \qquad \xi \in \cH, y \in R. \label{EqnPairingGNS}
\end{equation}
\end{nt}

\begin{comment} Our results and notations can be summarized by means of the following diagram.
\begin{figure}[h!]\label{FigLpIzu}
\[
%    \xymatrix{
% & M_\ast\ar@{^{(}->}[dr]^{\nu_1^{(z)}} & \\
%L_{(z)} \ar@{^{(}->}[ur]^{\mu_1^{(z)}}\ar@{^{(}->}[dr]_{\mu_\infty^{(z)}}&& %L_{(-z)}^\ast; \\
%& M \ar@{^{(}->}[ur]_{\nu_\infty^{(z)}}&} 
%\quad
 \xymatrix{
 & \qquad & M_\ast\ar@{^{(}->}[ddrr]  & \qquad &\\
 &  & L^p(M)_{{\rm left}} \ar@{^{(}->}[drr]  && \\
   L 
\ar@{^{(}->}[urr]^{l_p}
\ar@{^{(}->}[uurr]^{l_1}
\ar@{^{(}->}[drr]_{l_q}
\ar@{^{(}->}[ddrr]_{l_\infty}
\ar@{^{(}->}[rr]^{\Lambda}& & 
\cH
\ar@{^{(}->}[rr]& & R^\ast; \\
 & &  L^q(M)_{{\rm left}}\ar@{^{(}->}[urr] & &\\
& &  M \ar@{^{(}->}[uurr]_{ }&&}  
\]
\caption{Izumi's $L^p$-spaces.} 
\end{figure}
Here the inclusions $M_\ast \hookrightarrow R^\ast$, $\cH \hookrightarrow R^\ast$, $M \hookrightarrow R^\ast$ are given by the pairings,
\[
\langle \omega, y \rangle_{R_\ast, R} = \omega(x), \langle \xi, y \rangle_{R_\ast, R} = \langle \xi, \Lambda(y) \rangle, \langle x, y \rangle_{R_\ast, R} = \varphi(yx),
\]
where $\omega \in M_\ast, \xi \in \cH, x \in M$ and $y \in R$.
\end{comment}

\section{Intersections of $L^p$-spaces}\label{SectIntersections}

As indicated in the Section \ref{SectLp}, the intersections of the various $L^p$-spaces depend on the interpolation parameter $z$ of Definition \ref{DfnL}. Here we study the intersections  of $L^1(M)_{{\rm left}}$ and $L^2(M)_{{\rm left}}$, as well as the intersections of $L^2(M)_{{\rm left}}$ and $L^\infty(M)_{{\rm left}}$. The spaces turn out to be   natural and well-known in the theory of locally compact quantum groups.  We use the intersections in order to apply the re-iteration theorem, see \cite{BerghLof}.

\begin{nt}\label{NtSect3}
In this section, any interpolation space should be understood with respect to the diagram in (\ref{EqnLpIzu}) for the parameter $z = -1/2$. Recall that we introduced short hand notation for this diagram in Notations \ref{NotMMast} and \ref{NotGNS}. Moreover, we {\it identified} $M_\ast, \cH$ and $M$ as subspaces of $R^\ast$ by means of the pairings (\ref{EqnPairingMast}), (\ref{EqnPairingGNS}) and (\ref{EqnPairingM}). Similarly, any intersection of two such spaces should be understood as an intersection within $R^\ast$. The notation can be summarized by means of the non-dotted arrows in the following diagram. The dotted part of the diagram is the main topic of the present section.
 
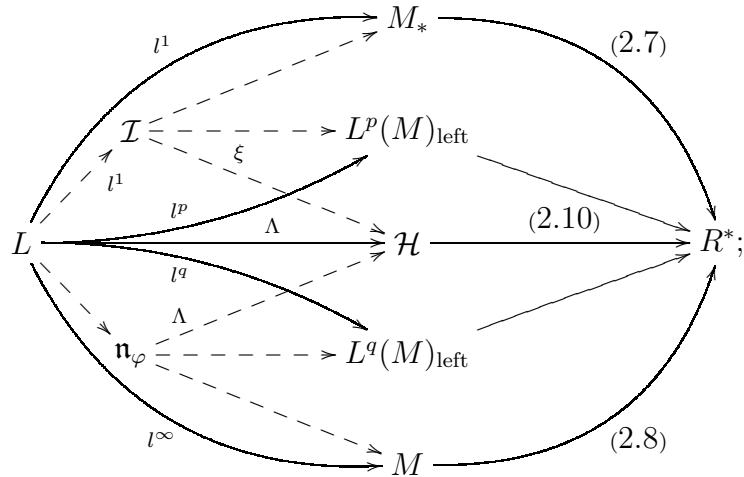
\[
%    \xymatrix{
% & M_\ast\ar@{^{(}->}[dr]^{\nu_1^{(z)}} & \\
%L_{(z)} \ar@{^{(}->}[ur]^{\mu_1^{(z)}}\ar@{^{(}->}[dr]_{\mu_\infty^{(z)}}&& %L_{(-z)}^\ast; \\
%& M \ar@{^{(}->}[ur]_{\nu_\infty^{(z)}}&} 
%\quad
 \xymatrix{
 & & \quad & M_\ast\ar@{->}@/^2.5pc/[ddrr]^{(\ref{EqnPairingMast})}  & \qquad &\\
 & \mathcal{I} \ar@{-->}[rru] \ar@{-->}[rrd]^(0.37){\xi}\ar@{-->}[rr]^{}& &L^p(M)_{{\rm left}} \ar@{->}[drr]  && \\
   L 
\ar@{->}@/_1pc/[urrr]^(0.4){l^p}
\ar@{->}@/^2.5pc/[uurrr]^{l^1}
\ar@{->}@/^1pc/[drrr]_(0.4){l^q}
\ar@{->}@/_2.5pc/[ddrrr]_{l^\infty}
\ar@{->}[rrr]^(0.65){\Lambda}
\ar@{-->}[ru]_(0.7){l^1}
\ar@{-->}[rd]
& & &
\cH
\ar@{->}[rr]^{(\ref{EqnPairingGNS})}& & R^\ast; \\
 & \nphi\ar@{-->}[rru]^(0.2){\Lambda} \ar@{-->}[rrd]\ar@{-->}[rr]^<{\qquad \qquad} &&  L^q(M)_{{\rm left}}\ar@{->}[urr] & &\\
& & & M \ar@{->}@/_2.5pc/[uurr]_{(\ref{EqnPairingM})}&&}  
\]
 
\end{nt}

\subsection{The intersection of $M_\ast$ and $\cH$} The following set defines the intersection of $M_\ast$ and $\cH$.

\begin{dfn}\label{DefI} We set: 
\[
\mathcal{I} = \left\{ \omega \in M_\ast \mid
\Lambda(x) \mapsto \omega(x^\ast), x \in \nphi \textrm{ is bounded
} \right\}.
\]
By the Riesz theorem, for every $\omega \in \mathcal{I}$,
there exists a $\xi(\omega) \in \cH$ such that 
$\langle \xi(\omega), \Lambda(x) \rangle = \omega(x^\ast)$. 
\end{dfn}

\begin{thm}\label{ThmIntersectionII}
We have $\mathcal{I} = \cH \cap  M_\ast$, where the equality should be interpreted within $R^\ast$, see Notation \ref{NtSect3}. Within $R^\ast$, $\omega \in \mathcal{I}$ equals $\xi(\omega) \in \cH$.
\end{thm}
\begin{proof}
We first prove $\supseteq$. Let $\xi \in \cH$ and $\omega \in
M_\ast$ be such that $\xi = \omega$ in $R^\ast$. For $y \in R$,
\[
 \omega(y) =\!\!\!\!\!\!\!\!^{(\ref{EqnPairingMast})} \langle \omega , y \rangle_{ R^\ast , R} =  \langle \xi, y \rangle_{ R^\ast , R} =\!\!\!\!\!\!\!\!\!^{(\ref{EqnPairingGNS})} \langle \xi, \Lambda(y^\ast) \rangle.
\]
$L$ contains $\mathcal{T}_\varphi^2$. Moreover, $\mathcal{T}_\varphi^2$ is a $\sigma$-strong-$\ast$/norm core for $\Lambda$, see Lemma \ref{LemCoreT}. Hence, it follows that $\omega \in \mathcal{I}$. 

To prove
$\subseteq$, let $\omega \in \mathcal{I}$. For $y \in R$,
\[ 
  \langle  \xi(\omega) , y \rangle_{ R^\ast , R} =\!\!\!\!\!\!\!\!\!^{(\ref{EqnPairingGNS})} \langle \xi(\omega), \Lambda(y^\ast) \rangle = \omega(y)  =\!\!\!\!\!\!\!\!^{(\ref{EqnPairingMast})}  \langle \omega,  y  \rangle_{R^\ast, R}  .
  \]
 Hence, $ \xi(\omega)  =  \omega $ in $R^\ast$.
\end{proof}

Note that $(M_\ast, \cH)$ forms a compatible couple. As explained in   Section \ref{SubSectInt},  the intersection of these two spaces carries a natural norm for which it is a Banach space. So, for
$\omega \in \mathcal{I}$ we define 
\[
\Vert \omega
\Vert_{\mathcal{I}} = \max \{ \Vert \omega \Vert, \Vert
\xi(\omega) \Vert \} .
\]

\begin{prop}\label{PropDensityII}
The map $k: L \rightarrow \mathcal{I}: x \mapsto
_x \!\! \varphi $ is injective, norm-decreasing and has dense
range. In fact, $k(\mathcal{T}_\varphi^2)$ is $\Vert \cdot \Vert_{\mathcal{I}}$-dense in $\mathcal{I}$. 
\end{prop}
\begin{proof}
 Suppose that $x \in L $ and $_x \varphi = 0$, then $0 = (_x \varphi  )(x^\ast) = \varphi(x^\ast x)$. So $x = 0$ and hence $k$ is injective. For $x \in L $, $ \Vert _x \varphi \Vert \leq \Vert x \Vert_{L}$ and
\[
  \Vert \xi(_x \varphi ) \Vert = \Vert \Lambda(x) \Vert = \Vert _x \varphi (x^\ast)\Vert^{1/2} \leq \Vert _x \varphi  \Vert^{1/2} \Vert x^\ast\Vert^{1/2} \leq \Vert x \Vert_{L },
\]
so that $k$ is norm-decreasing. Now we prove that the range of $k$
is dense in $\mathcal{I}$. We identify $\mathcal{I}$ with the
subspace $\{ (\omega, \xi(\omega) ) \mid \omega \in \mathcal{I} \}
\subseteq M_\ast \times \cH$. We equip $M_\ast \times \cH$ with
the norm $\Vert (\omega, \xi) \Vert_{{\rm max}} = \max\{ \Vert \omega \Vert,
\Vert \xi \Vert \}$. The norm coincides with $\Vert \cdot
\Vert_{\mathcal{I}}$ on $\mathcal{I}$. The dual of $(M_\ast \times
\cH, \Vert \cdot \Vert_{\max})$ can be identified with $(M \times
\cH^\ast, \Vert \cdot \Vert_{{\rm sum}}) $, where $\Vert (x, \xi)
\Vert_{{\rm sum}} = \Vert x \Vert + \Vert \xi \Vert$. Let $N
\subseteq M \times \cH^\ast$ be the space of all $(y, \eta)$ such that
$\langle (\omega, \xi(\omega)), (y, \eta) \rangle_{M_\ast \times
\cH, M \times \cH^\ast } = 0$ for all
$\omega \in \mathcal{I}$. The dual of $\mathcal{I}$ is given by
$(M \times \cH) / N$ equipped with the quotient norm. 

Now, let $(y, \eta) \in M \times \cH$ be such that 
\[
\langle
(_x \varphi , \Lambda(x)), (y, \eta) \rangle_{M_\ast \times
\cH, M \times \cH^\ast } = (_x \varphi)(y) + \langle \Lambda(x), \eta\rangle = 0
\]
 for all $x \in \mathcal{T}_\varphi^2 $. The proof is finished if we can show that $(y,
\eta) \in N$. In order to do this, let $(e_j)_{j \in J}$ be a net as in Lemma \ref{LemApprox}. Put $a_j = \sigma_{-\frac{i}{2}}(e_j)$. By the assumptions
on $(y, \eta)$, for $x \in \mathcal{T}_\varphi$,
\begin{equation}\label{EqnZeroFunctional}
 (_{xa_j}\varphi )(y) = - \langle \Lambda(xa_j), \eta \rangle.
\end{equation}
For the left hand side we find by Corollary \ref{CorInclusions},
\begin{equation}\label{EqnZeroFunctionalLHS}
  (_{xa_j}\varphi )(y) = \varphi(\sigma_i(a_j) yx) = \langle \Lambda(x),
\Lambda(y^\ast \sigma_{i}(a_j)^\ast) \rangle,
\end{equation}
where the first equality follows from \cite[Proposition 2.3]{Izumi}. For the right hand side of (\ref{EqnZeroFunctional}) we find
\begin{equation}\label{EqnZeroFunctionalRHS}
 \langle \Lambda(xa_j), \eta \rangle = \langle J  \sigma_{-\frac{i}{2}}(a_j^\ast)  J \Lambda(x) , \eta \rangle =
\langle \Lambda(x) ,  J  \sigma_{\frac{i}{2}}(a_j)  J
\eta \rangle.
\end{equation}
Hence (\ref{EqnZeroFunctional}) together with
(\ref{EqnZeroFunctionalLHS}) and (\ref{EqnZeroFunctionalRHS})
yield
\[
\Lambda(y^\ast \sigma_{i}(a_j)^\ast) = - J
 \sigma_{\frac{i}{2}}(a_j)  J  \eta.
\]
Hence, since $\sigma_{\frac{i}{2}}(a_j) = e_j \rightarrow 1$
strongly, $\Lambda(y^\ast \sigma_{i}(a_j)^\ast) \rightarrow -
\eta $ weakly. For $\omega \in \mathcal{I}$,
\[
 \langle \xi(\omega), \eta \rangle =  - \lim_{j \in J} \langle \xi(\omega),  \Lambda(y^\ast \sigma_{i}(a_j)^\ast) \rangle = - \lim_{j \in J} \omega(\sigma_{i}(a_j) y) =  - \lim_{j \in J} \omega(\sigma_{\frac{i}{2}}(e_j) y) = -\omega(y).
\]
Thus $(y, \eta) \in N$.
\end{proof}

\subsection{The intersection of $\cH$ and $M$}

 It turns out that $\nphi$ is
the intersection of $M$ and $\cH$.

\begin{thm} \label{ThmIntersectionIII}   We have 
$\nphi = \cH \cap M$, where the equality should be interpreted within $R^\ast$, see Notation \ref{NtSect3}. Within $R^\ast$, $x\in \nphi$ equals $\Lambda(x) \in \cH$.

Moreover, let $\overline{L}$ be the closure of $l^\infty(L)$ in $M$. Then $\nphi = \cH \cap \overline{L}$. 
\end{thm}
\begin{proof}
First we prove that $\nphi = \cH \cap
M$ in $R^\ast$. For $x \in \nphi, y \in R$, 
\[
 \langle \Lambda(x)  , y \rangle_{R^\ast, R} =\!\!\!\!\!\!\!\!\!^{(\ref{EqnPairingGNS})}  \langle \Lambda(x), \Lambda(y^\ast) \rangle   =  \varphi(yx) =\!\!\!\!\!\!\!\!^{(\ref{EqnVarphiLeftRight})} \varphi_y(x) =\!\!\!\!\!\!\!\!^{(\ref{EqnPairingM})} \langle x,  y \rangle_{R^\ast, R},
 \]
so $  \Lambda(x)  =  x$ in $R^\ast$. Hence
the inclusion $\subseteq$ follows. 

Now let $x \in M$, $\xi \in \cH$ be such that
$ x =  \xi $ in $R^\ast$.   For $y \in R$,
\[
\varphi_y(x) =\!\!\!\!\!\!\!\!^{(\ref{EqnPairingM})} \langle x, y \rangle_{R^\ast, R} =    \langle \xi, y \rangle_{R^\ast, R} =\!\!\!\!\!\!\!\!\!^{(\ref{EqnPairingGNS})} \langle \xi, \Lambda(y^\ast) \rangle.
\]
  For $a \in \mathcal{T}_\varphi^2$, $y \in \nphi$, using Corollary \ref{CorInclusions} for the third, fourth and fifth equality,
%The 3rd equality really relies on \ref{CorInclusions}, since one does not know if $y^\astx \in \mathfrak{m}_\varphi$ a la Kustermans, notes Prop 4.25! 
\[
 \begin{split}
 \\ & \langle \Lambda(xa), \Lambda(y) \rangle = \varphi(y^\ast xa) = _a \varphi (y^\ast x) = \varphi_{\sigma _{i}(a)} (y^\ast x) \\ = &  \varphi_{\sigma _{i}(a)y^\ast} (x)  = \langle \xi, \Lambda(y \sigma _i(a)^\ast ) \rangle = \langle J \sigma_{i/2}(a)^\ast J  \xi, \Lambda(y) \rangle.
 \end{split}
\]
So for $a \in \mathcal{T}_\varphi^2$, $\Lambda(xa) =  J \sigma_{i/2}(a)^\ast J  \xi$. Let $(e_j)_{j \in J}$ be a net as in Lemma \ref{LemApprox}. Put $a_j = e_j^2$. Then $x a_j \rightarrow x$ $\sigma$-weakly. Furthermore, $J \pi(\sigma_{i/2}(a_j)^\ast) J \xi \rightarrow \xi$ weakly, hence $\Lambda(xa_j)$ converges weakly. Since $\Lambda$ is $\sigma$-weak/weak closed, $x \in \Dom(\Lambda) = \nphi$ and $\xi = \Lambda(x)$. This proves $\supseteq$.

Recall the complex interpolation method from Definition \ref{DfnMorphCpt}. Recall that in this section every interpolation space should be interpreted with respect to (\ref{EqnLpIzu}) for parameter $z = -1/2$. 
Note that \cite[Theorem 4.2.2]{BerghLof} gives the second equality in 
\begin{equation}\label{EqnIntI}
\cH = (M, M_\ast)_{[1/2]} = (\overline{L}, M_\ast)_{[1/2]}  \subseteq \overline{L} + M_\ast.
\end{equation}
We now prove that
\begin{equation}\label{EqnIntII}
M \cap (\overline{L} + M_\ast) = \overline{L}.
\end{equation}
Take any $s \in M \cap (\overline{L} + M_\ast) \subseteq R^\ast$. Since $s \in  \overline{L} + M_\ast $, we can choose representatives $x \in \overline{L}, \omega \in M_\ast$ such that $s = x + \omega$ in $R^\ast$. Since $s \in M$, we can find a representative $y \in M$ such that $s = y$ in $R^\ast$. Then $\omega = y - x$ is both in $M_\ast$ and $M$, and hence by (\ref{EqnIzumiIntersection}) in $M_\ast \cap M = L$. Hence we see that $s = x + \omega \in \overline{L} + L = \overline{L}$. This proves $\subseteq$, the other inclusion is trivial.
   
Now, (\ref{EqnIntI}) and (\ref{EqnIntII}) imply:
\[
\cH \cap M  = \cH \cap M \cap (\overline{L} + M_\ast)= \cH \cap \overline{L}.
\] 
\end{proof}

Again, we introduce the norm on an intersection of a compatible couple as in Section \ref{SubSectInt}. 
For $x \in \nphi$, we put 
\[
\Vert x \Vert_{\nphi} = \max \{ \Vert x
\Vert, \Vert \Lambda(x) \Vert \}.
\]
Again, we can prove a density result similar to Propostion \ref{PropDensityII}

\begin{prop}\label{PropDensityIII}
The map $k': L \rightarrow \nphi: x \mapsto x$ is
injective, norm-decreasing and has dense range.
\end{prop}
 \begin{proof}
The non-trivial part is that $k'(L)$ is dense in $\nphi$
with respect to $\Vert \cdot \Vert_{\nphi}$. To prove this, we
identify $\nphi$ with the subspace $\{ (x, \Lambda(x)) \mid x \in
\nphi   \} \subseteq M \times \cH$. For $(x, \xi) \in M \times
\cH$, we set $\Vert (x, \xi) \Vert_{\max} = \max \{ \Vert x \Vert,
\Vert \xi \Vert \} $. So $\Vert \cdot \Vert_{\max}$ coincides with
$\Vert \cdot \Vert_{\nphi}$ on $\nphi$. The dual of $(M \times
\cH, \Vert \cdot \Vert_{\max})$  is given by $(M^\ast \times \cH^\ast,
\Vert \cdot \Vert_{{\rm sum}})$, where $\Vert(\theta,
\xi)\Vert_{{\rm sum}} = \Vert \theta \Vert + \Vert \xi \Vert$.

\vspace{0.3cm}

Let $(\theta, \xi) \in M^\ast \times \cH^\ast$ be such that for all $x
\in L$,
\begin{equation}\label{EqnDensityCondition}
\theta(x) + \langle \Lambda(x), \xi \rangle = 0.
\end{equation}
We must prove that (\ref{EqnDensityCondition}) holds for all $x
\in \nphi$. The proof proceeds in several steps.

% We will first prove that (\ref{EqnDensityCondition}) holds for $x \in \nphi \cap \nphi^\ast$.

\vspace{0.3cm}

{\sc Claim I:} There exists an $\omega \in M_\ast$ such that
for $x \in L  \cap R$, $\omega(x) = \theta(x)$.

\noindent {\it Proof of the claim.} From Corollary \ref{CorInclusions} it follows that 
\[
\overline{L \cap R} \left(= \overline{l^\infty(L) \cap r^\infty(R)}\right)
\]
 is a C$^\ast$-algebra.  Here and in the rest of this proof the closure has to be interpreted within $M$. Let $(u_j)_{j \in J}$ be an approximate unit for the
C$^\ast$-algebra $\overline{L \cap R}$. We may
assume that $u_j \in (L \cap R)^+$. Set
$\omega_j(x) = -\langle x \Lambda(u_j), \xi \rangle, x \in M$. So $\omega_j \in M_\ast$.   Moreover, by (\ref{EqnDensityCondition}) and Corollary \ref{CorInclusions},
\[
\omega_j(x) =  -\langle x \Lambda(u_j), \xi \rangle = \theta(xu_j).
\]
 
Let $\rho$ be a representation of $\overline{ L \cap
R}$ on a Hilbert space $\cH_\rho$ such that $\theta(x) =
\langle \rho(x) \xi, \eta \rangle$ for certain vectors $\xi, \eta
\in \cH_\rho$. Then $\omega_j(x) = \langle \rho (x) \rho(u_j) \xi,
\eta \rangle$. Since $\rho(u_j) \rightarrow 1$ strongly, $\Vert \omega_j
\vert_{\overline{L \cap R}} - \theta
\vert_{\overline{L \cap R}}\Vert \rightarrow 0$. $ L \cap R (\supseteq \mathcal{T}_\varphi^2)$ is $\sigma$-weakly,
hence strongly dense in $M$ so that by Kaplansky's density theorem
$\Vert \omega_j \Vert = \Vert \omega_j \vert_{L \cap R} \Vert$. Hence $(\omega_j)_{j \in J}$ is a Cauchy net in $M_\ast$.
Let $\omega \in M_\ast$ be its limit. This proves the first the claim.

\vspace{0.3cm}

{\sc Claim II:} For $x \in \nphi$, we find $\omega(x)
=   - \langle \Lambda(x), \xi \rangle$.

\noindent {\it Proof of the claim.}
Note that $L \cap R$ is a $\sigma$-weak/weak
core for $\Lambda$. Indeed, $\mathcal{T}_\varphi^2$ is contained in $L \cap R$ so that we can apply Lemma \ref{LemCoreT}.

Now, if $x \in L \cap R$, the claim follows by the first claim and the properties of $\theta$, i.e. $\omega(x) = \theta(x) = -
\langle \Lambda(x), \xi \rangle$. Let $x \in \nphi$ and, by the previous paragraph, let
$(x_i)_{i \in I}$ be a net in $L \cap R$ converging $\sigma$-weakly to $x$ such
that $\Lambda(x_i) \rightarrow \Lambda(x)$ weakly. Then, we arrive at the following equation:
\begin{equation}\label{EqnL1L2Comparison}
\omega(x)
= \lim_{i \in I} \omega(x_i) = - \lim_{i \in I} \langle
\Lambda(x_i), \xi \rangle = - \langle \Lambda(x), \xi \rangle.
\end{equation}
This proves the second claim.

\vspace{0.3cm}

{\sc Claim III:}  $\overline{L
\cap R} = \overline{\nphi \cap \nphi^\ast} =
\overline{L \cap R}$, where the closures are interpreted with respect to the norm on $M$.

\noindent {\it Proof of the claim.} Note that by Proposition \ref{PropAltDescr},
$\overline{L  \cap R} \subseteq \overline{\nphi
\cap \nphi^\ast}$.
By Theorem \ref{ThmIntersectionIII}, we see that $\nphi \subseteq \overline{L}$. Since $R = \{ x^\ast \mid x \in L\}$, see Corollary \ref{CorInclusions}, we also have $\nphi^\ast \subseteq \overline{R}$. Hence,
\[
\overline{L  \cap R} \subseteq \overline{\nphi
\cap \nphi^\ast} \subseteq \overline{L} \cap \overline{R} \qquad \textrm{ (closures in } M {\rm )}.
\]
The inclusions are in fact equalities. Indeed, let $x \in \overline{L } \cap
\overline{R}$ be positive. Let $x_n$ and $y_n$ be sequences in
$L$, respectively $R$, converging in norm to $x$.
Then, by Corollary \ref{CorInclusions}, $y_n x_n \in LR \subseteq L \cap
R$. $y_n x_n$ is norm convergent to $x^2$. So $x^2 \in
\overline{ L \cap R}$, hence $x \in
\overline{L \cap R}$. From Corollary \ref{CorInclusions} it follows that $\overline{ L \cap R}$ and $\overline{L \cap R}$ are C$^\ast$-algebras. Hence,  $\overline{L
\cap R} = \overline{\nphi \cap \nphi^\ast} =
\overline{L \cap R}$.

\vspace{0.3cm}

{\sc Claim IV:} Equation (\ref{EqnDensityCondition}) holds for $x \in \nphi \cap \nphi^\ast$.

\noindent{\it Proof of the claim.}
Let $x \in \nphi \cap \nphi^\ast$ and by the third claim, let $x_n \in L  \cap
R$ be a sequence converging in norm to $x$. Then, using the first claim in the second equality and the third claim in the fourth equality,
\[
\theta(x) = \lim_{n \rightarrow \infty} \theta(x_n) = \lim_{n
\rightarrow \infty} \omega(x_n) = \omega(x) = - \langle \Lambda(x)
, \xi \rangle.
\]
 Hence (\ref{EqnDensityCondition}) follows for $x
\in \nphi \cap \nphi^\ast$. 

\vspace{0.3cm}

\noindent{\it Proof of the proposition.} Let $x \in \nphi$ and let $x = u
\vert x \vert$ be its polar decomposition. Since (\ref{EqnDensityCondition}) holds for $y \in L$, we find for $y \in
L$ that $uy \in L$ by Corollary \ref{CorInclusions} and,
\begin{equation}\label{EqnDensityConditionPrime}
 (\theta u)(y) + \langle \Lambda(y), u^\ast \xi \rangle = 0,
\end{equation}
where we defined $\theta  u \in M$ by $(\theta   u)(a) = \theta(ua), a \in M$. 
If we apply the arguments in the previous paragraphs to the pair $(\theta   u, u^\ast \xi )$, we see that actually (\ref{EqnDensityConditionPrime}) holds for
all $y \in \nphi \cap \nphi^\ast$. In particular, putting $y =
\vert x \vert$, the required equation (\ref{EqnDensityCondition})
follows.
 \end{proof}

\subsection{Re-iteration}

Here we apply the re-iteration theorem, see \cite[Theorems 4.6.1]{BerghLof}, for the
complex interpolation method to obtain $L^p(M)_{{\rm left}}, p \in (1,2]$ as an interpolation space of $\cH$ and $M_\ast$. Similarly, $L^p(M)_{{\rm left}}, p \in [2,\infty)$ as an interpolation space of $\cH$ and $M$. Recall that in this section every intersection and interpolation is understood with respect to (\ref{EqnLpIzu}) for the parameter $z = -1/2$.

\begin{thm}\label{ThmInterpolationL2} 
We have the following interpolation properties:
\begin{enumerate}\item\label{ItemIntI} For $p \in (1, 2]$, $(\cH, M_\ast)_{[\frac{2}{p} - 1]} = L^p(M)_{{\rm left}}$. 
\item\label{ItemIntII} For $q \in [2, \infty)$,
 $(\cH, M)_{[1-\frac{2}{q}]} = (M, \cH)_{[\frac{2}{q}]} =  L^q(M)_{{\rm left}}$.
\end{enumerate}
\end{thm}
\begin{proof}
(\ref{ItemIntI}) 
Recall that $\overline{L} = \overline{l^\infty(L)}$ denotes the closure of $L$ in $M$. Recall from (\ref{EqnIzumiIntersection}) that $M_\ast \cap M = L$.
By \cite[Theorem 4.2.2 (b)]{BerghLof} we get the first equality in of:
\begin{equation}
(\overline{L},
M_\ast)_{[\frac{1}{2}]} = (M, M_\ast)_{[\frac{1}{2}]} = L^2(M)_{{\rm left}} \simeq \cH.
\end{equation}
The latter isomorphism is the identification in Notation \ref{NotGNS}. On the other hand, we find:
\begin{equation}
(\overline{L},
M_\ast)_{[1]} = (M, M_\ast)_{[1]} \label{EqnRightInterpolation}
\end{equation}
Since $l^1(L)$ is dense in $M_\ast$, see \cite[Proposition 2.4]{Izumi},  we find that (\ref{EqnRightInterpolation}) in turn equals $M_\ast$ by \cite[Proposition 4.2.2]{BerghLof}.
 
Note that the following three density assumptions are satisfied:
\begin{enumerate}[(i)]
\item $l^1(L)$ is dense in $M_\ast$, see \cite[Proposition 2.4]{Izumi}.
\item $l^\infty(L)$ is dense in $\overline{L}$ (trivial).
\item $l^1(L)$ is $\Vert \cdot \Vert_{\mathcal{I}}$-dense in $\mathcal{I}$ by Proposition \ref{PropDensityII}. Moreover $\mathcal{I}$ is the intersection of $M_\ast$ and $\cH$ in $R^\ast$, see Theorem \ref{ThmIntersectionII}.
\end{enumerate}
 Hence, we have checked the assumptions of the re-iteration theorem \cite[Theorems 4.6.1]{BerghLof} which is used in the third equality,
 \[ 
  L^p(M )_{{\rm left}} =
 (M, M_\ast)_{[\frac{1}{p}]} =
 (\overline{L}, M_\ast)_{[\frac{1}{p}]}   =  
  ((\overline{L}, M_\ast)_{[\frac{1}{2}]}, (\overline{L}, M_\ast)_{[1]})_{[\frac{2}{p}-1]}  =
 (\cH , M_\ast)_{[\frac{2}{p}-1]} , 
 \]
(here the second equality follows again by \cite[Theorem 4.2.2]{BerghLof}).
\begin{comment}
 Similarly, let $q \in [2, \infty)$. Using \cite[Theorems 4.2.2 and
4.6.1]{BerghLof}, Theorem \ref{ThmInterpolationProp}, Lemma
\ref{LemDensityIII}, Propositions \ref{PropIntersection} and
\ref{PropIntersectionIII} and the density of the range of $\mu_2$,
we find:
\[
 L^q = (M, M_\ast)_{[\frac{1}{q}]} = ( \overline{\mu_\infty(L_0)}, M_\ast)_{[\frac{1}{p}]} = ((\overline{\mu_\infty(L_0)}, M_\ast)_{[0]}, (\overline{\mu_\infty(L_0)}, M_\ast)_{[\frac{1}{2}]})_{[\frac{2}{q}]} =
(\overline{\mu_\infty(L_0)}, L^2)_{[\frac{2}{q}]}.
\]

$(L^2, M)^V_{[1-\frac{2}{q}]} = (M, L^2)^V_{[\frac{2}{q}]}$
follows from \cite[Theorem 4.2.1]{BerghLof}.
\end{comment}

(\ref{ItemIntII}) Completely analogously, using Theorem \ref{ThmIntersectionIII} and Proposition \ref{PropDensityIII}, one proves that 
\[
L^q(M)_{{\rm left}} = (\cH,
M)_{[1-\frac{2}{q}]},
\]
 which in turn equals $(M, \cH)_{[\frac{2}{q}]}$ by \cite[Theorem 4.2.1]{BerghLof}.
\end{proof}

\section{Locally compact quantum groups}\label{SectLcqg} 
We now recall the Kustermans-Vaes definition of a locally compact quantum groups, see \cite{KusV} and \cite{KusVII}.  Since we will be dealing with non-commutative $L^p$-spaces, we stick to the von Neumann algebra setting. For an introduction to the theory of locally compact quantum groups  we refer to \cite{KusLec} or \cite{Tim}, where the results below are summarized. See also \cite{DaeleLcqg} were a simple von Neumann algebraic approach to quantum groups is presented. 

\subsection{Von Neumann algebraic quantum groups} 

\begin{dfn}
A locally compact quantum group $(M, \Delta)$ consists of the
following data:
\begin{enumerate}
 \item A von Neumann algebra $M$;
\item A unital, normal $\ast$-homomorphism $\Delta: M \rightarrow M \otimes M$ satisfying the coassociativity relation $(\Delta \otimes \iota) \circ \Delta = (\iota \otimes \Delta) \circ \Delta$, where $\iota: M \rightarrow M$ is the identity;
\item  Two  normal, semi-finite, faithful weights
$\varphi$, $\psi$ on $M$ so that
\[
\begin{split}
\varphi \left( (\omega \otimes \iota )\Delta(x)\right)\, &=\,
\varphi(x)\omega(1), \qquad \forall \ \omega \in M^+_*,\, \forall\
x\in \mathfrak{m}^+_\varphi
\qquad \text{(left invariance);}\\
\psi\left( (\iota \otimes\omega)\Delta(x)\right)\, &=\,
\psi(x)\omega(1), \qquad \forall \ \omega \in M^+_*,\, \forall\
x\in \mathfrak{m}^+_\psi \qquad \text{(right invariance)}.
\end{split}
\]
$\varphi$ is the left Haar weight and $\psi$ the right Haar
weight.
\end{enumerate}
\end{dfn}

 Note that we suppress the Haar weights in the notation. 
 The triple $(\cH, \pi,
\Lambda)$ denotes the GNS-construction with respect to the left
Haar weight $\varphi$. We may assume that $M$ acts on the
GNS-space $\cH$. 

In order to reflect to the classical situation of a locally compact group, we include the following example. 

\begin{example}\label{ExampleClassical}
Let $G$ be a locally compact group. Consider $M = L^\infty(G)$ and define the coproduct $\Delta_G: L^\infty(G) \rightarrow L^\infty(G) \otimes L^\infty(G) \simeq L^\infty(G \times G)$ by putting 
\[
(\Delta_G(f))(x,y) = f(xy).
\]
  $\varphi$ and $\psi$ are given by integrating against the left and right Haar weight respectively. In this way $(L^\infty(G), \Delta_G)$ is a locally compact quantum group.
\end{example}

\subsection{Multiplicative unitary}

There exists a unique unitary operator $W\in
B(\cH \otimes \cH)$ defined by
\[
W^\ast \left( \Lambda (a)\otimes  \Lambda(b) \right) = \left(
\Lambda \otimes \Lambda \right) \left( \Delta (b)(a\otimes
1)\right), \qquad a,b \in \nphi.
\]
$W$ is known as the multiplicative unitary. It satisfies the
pentagonal equation $W_{12}W_{13}W_{23}=W_{23}W_{12}$ in $B(\cH
\otimes \cH \otimes \cH)$. Furthermore, $\Delta(x) = W^\ast (1 \otimes x) W, x \in M$.

\subsection{The dual quantum group}
In \cite{KusV}, \cite{KusVII}, it is
proved that there exists a dual locally compact quantum group
$(\hat{M},\hat{\Delta})$, so that
$(\hat{\hat{M}},\hat{\hat{\Delta}}) = (M,\Delta)$. The dual left
and right Haar weight are denoted by $\hat{\varphi}$ and
$\hat{\psi}$. Similarly, all other dual objects will be denoted by
a hat. By construction,
\[\hat{M} = \overline{\left\{ (\omega
\otimes \iota)(W) \mid \omega \in B(\cH)_\ast
\right\}}^{\sigma{\rm -strong-}\ast}.
\] Furthermore, $\hat{W} =
\Sigma W^\ast \Sigma$, where $\Sigma$ denotes the flip on $\cH
\otimes \cH$. This implies that $W \in M \otimes \hat{M}$ and
\[
M = \overline{\left\{ (\iota \otimes \omega)(W) \mid \omega \in
B(\cH)_\ast \right\}}^{\sigma{\rm -strong-}\ast}.
\]
The dual coproduct can be given by the dualized formula $\hat{\Delta}(x) = \hat{W}^\ast (1 \otimes x) \hat{W}, x \in \hat{M}$. 
 For $\omega \in
M_\ast$, we use the standard notation 
\[
\lambda(\omega) = (\omega
\otimes \iota) (W).
\]

Finally, we introduce the dual left Haar weight $\hat{\varphi}$. 
Recall from Definition \ref{DefI}, that we let $\mathcal{I}$ be the set of
$\omega \in M_\ast$, such that $\Lambda(x) \mapsto \omega(x^\ast),
x \in \nphi$ extends to a bounded functional on $\cH$.   By the
Riesz theorem, for every $\omega \in \mathcal{I}$, there is a
unique vector denoted by $\xi(\omega) \in \cH$ such that
$\omega(x^\ast) = \langle \Lambda(x) , \xi(\omega)\rangle, x \in
\nphi$.

\begin{dfn}
 The dual left Haar weight $\hat{\varphi}$ is defined to be
the unique normal, semi-finite, faithful weight on $\hat{M}$, with
GNS-construction $(\cH, \iota, \hat{\Lambda})$ such that
$\lambda(\mathcal{I})$ is a $\sigma$-strong-$\ast$/norm core for
$\hat{\Lambda}$ and $\hat{\Lambda}(\lambda(\omega)) = \xi(\omega),
\omega \in \mathcal{I}$.
\end{dfn}

Since we do not need it, we merely mention that there also exists a dual right Haar weight. The following example gives the dual structure in the classical situation.

\begin{example}\label{ExampleClassicalDual}
Let $G \rightarrow B(L^2(G)): x \mapsto \lambda_x$ be the left regular representation.
For $(M, \Delta_G)$ as in Example \ref{ExampleClassical}, one finds that $\cH \otimes \cH = L^2(G) \otimes L^2(G) \simeq L^2(G \times G)$ and
\[
W f(x,y)  = f(x, x^{-1}y).
\]
For $f \in L^1(G)$, let $\omega_f$ be the functional on $L^\infty(G)$ defined by $\omega_f(g) = \int_G f(x) g(x) d_lx$. Then, 
\[
\lambda(\omega_f) = (\omega_f \otimes \iota)(W) =  \int_G f(x) \lambda_x d_lx, 
\]
where the integral is in the $\sigma$-strong-$\ast$ topology. So $\lambda$ is the left regular representation. We find that
$\hat{M}$ is given by the group von Neumann algebra $\hat{M} = \mathcal{L}(G)$. 

 For completeness, we mention that $\hat{\Delta}(\lambda_x) = \lambda_x \otimes \lambda_x$. The dual left Haar weight is given by the Plancherel weight \cite{TakII}. For a continous, compactly supported function $f$ on $G$, one finds $\hat{\varphi}(\lambda(f)) = f(e)$, where $e$ is the identity of $G$.

If $G$ is abelian, conjugation with the $L^2$-Fourier transform shows that this structure is isomorphic to $(L^\infty(\hat{G}), \Delta_{\hat{G}})$.
\end{example}

\section{Fourier theory}\label{SectFT}

In this Section we define an $L^p$-Fourier transform. Our strategy is  similar to the one defining the classical $L^p$-Fourier transform on locally compact abelian groups. We first define a $L^1$- and $L^2$-Fourier transform and show that they form a compatible pair of morphisms, see Remark \ref{RmkCptMor}. Then we apply the complex interpolation method to get a bounded $L^p$-Fourier transform for $p \in [1,2]$,
\[
\mathcal{F}_p: L^p(M)_{{\rm left}} \rightarrow L^q(M)_{{\rm left}}, \qquad    \frac{1}{p} + \frac{1}{q} = 1.
\]
The crucial property of the complex interpolation method is the one given by Theorem \ref{ThmCplxInterpolationBound}. The theorem gives the non-commutative analogue of the Riesz-Thorin theorem as mentioned in the introduction. It is for this reason that we have approached $L^p$-spaces from the perspective of interpolation spaces and that we have used Izumi's definition. 

\begin{nt}\label{NtSect5}
From now on, let $(M, \Delta)$ be a locally compact quantum group with left Haar weight $\varphi$. $(\hat{M}, \hat{\Delta})$ is the Pontrjagin dual. In this section all $L^p$-spaces we encounter are `left' $L^p$-spaces which are defined with respect to the (dual) left Haar weight. More precisely, we stick to Notation \ref{NtSect3}. We equip the objects introduced in Section \ref{SectLp} with a  hat if they are associated with the dual quantum group. So we get $\hat{L}, \hat{R}, L^p(\hat{M})_{{\rm left}}, \ldots$ Recall that by construction $\hat{\cH} = \cH$.
\end{nt}

\begin{thm}[$L^1$- and $L^2$-Fourier transform]\label{ThmFT}
We can define compatible Fourier transforms in the following way:
\begin{enumerate}
\item\label{ItemFTI} There exists a unique unitary map $\mathcal{F}_2: \cH \rightarrow \cH$, which is determined by:
\begin{equation}\label{EqnL2Transform}
\Lambda(x) \mapsto \hat{\Lambda}(\lambda(_x \varphi)), \qquad x \in L.
\end{equation}
\item\label{ItemFTII} There exists a bounded map $\mathcal{F}_1: M_\ast \rightarrow \hat{M}: \omega \mapsto \lambda(\omega)$. Moreover, $\Vert \mathcal{F}_1 \Vert = 1$.
\item\label{ItemFTIII} $\mathcal{F}_2: \cH \rightarrow \cH$ and  $\mathcal{F}_1: M_\ast \rightarrow \hat{M}$ are compatible in the sense of Remark \ref{RmkCptMor}, i.e. the following diagram commutes:
\begin{equation}\label{EqnDiagL1L2}
\xymatrix{
&M_\ast \ar[dr]^{(\ref{EqnPairingMast})} \ar@2[rrrr]^{\mathcal{F}_1} &&&&\hat{M} \ar[dr]^{(\ref{EqnPairingM})} &\\
L \ar[r]^{\Lambda}\ar[ru]^{l^1} & \cH \ar[r]^{(\ref{EqnPairingGNS})} \ar@2@(dr,dl)[rrrr]_{\mathcal{F}_2} & R^\ast & & \hat{L}\ar[ru]^{\hat{l}^\infty} \ar[r]^{\hat{\Lambda}} & \cH \ar[r]^{(\ref{EqnPairingGNS})} & \hat{R}^\ast.
}
\end{equation}
\end{enumerate}
\end{thm}
\begin{proof}
(\ref{ItemFTI}) By Proposition \ref{ThmIntersectionII}, we see that for $x \in L$, we have $_x \varphi \in \mathcal{I}$. By definition of $\hat{\Lambda}$, we have $\hat{\Lambda}(\lambda(_x \varphi)) = \xi(_x \varphi)$. Since by (\ref{EqnVarphiLeftRight}),
\[
_x \varphi(y^\ast) = \varphi(y^\ast x ) = \langle \Lambda(x), \Lambda(y) \rangle, \qquad y \in \nphi, 
\]
we see that by definition of $\xi(_x \varphi)$, we have $\xi(_x \varphi) = \Lambda(x)$. So (\ref{EqnL2Transform}) is the identity map. Since $\mathcal{T}_\varphi ^2 \subseteq L$ and $\Lambda(\mathcal{T}_\varphi ^2 )$ is dense in $\cH$, see the much stronger result of Lemma \ref{LemCoreT}, this determines a map on $\cH$. 

(\ref{ItemFTII}) The norm bound follows, since:
\[
\Vert \lambda(\omega) \Vert = \Vert (\omega \otimes \iota)(W)\Vert \leq \Vert (\omega \otimes \iota) \Vert \Vert W \Vert \leq \Vert \omega \Vert.
\]

(\ref{ItemFTIII}) For $y \in \hat{R}, x \in L$, we find:
\[
\begin{split}
&\langle \mathcal{F}_2 \Lambda(x)  , y \rangle_{\hat{R}^\ast, \hat{R}} = 
\langle \hat{\Lambda}( \lambda(_x \varphi))  , y \rangle_{\hat{R}^\ast, \hat{R}} = \!\!\!\!\!\!\!\!^{(\ref{EqnPairingGNS})}
\langle \hat{\Lambda}( \lambda(_x \varphi))  , \hat{\Lambda}( y^\ast) \rangle  \\ =& \hat{\varphi}(y \lambda(_x \varphi)) =\!\!\!\!\!\!\!\!^{(\ref{EqnVarphiLeftRight})}  \hat{\varphi}_y (\lambda(_x \varphi)) =\!\!\!\!\!\!\!\!^{(\ref{EqnPairingM})} \langle  \lambda(_x \varphi) , y \rangle_{\hat{R}^\ast, \hat{R}} = \langle \mathcal{F}_1(_x \varphi), y \rangle_{\hat{R}^\ast, \hat{R}},
\end{split}
\]
which proves the commutativity of the diagram.
\end{proof}

\begin{rmk}
Kahng \cite{KahngFourier} defines an operator algebraic Fourier transform and in principle the idea behind Theorem \ref{ThmFT} can also be found here. However, \cite[Definition 3]{KahngFourier} has to be given a more careful interpretation, since, if $\varphi$ is not a state, the expression $(\varphi \otimes \iota)(W (a \otimes 1)), a \in \hat{\lambda}(\hat{\mathcal{I}})$, is in general undefined. In case $\varphi$ is a state, our definition of $\mathcal{F}_2$ equals Kahng's by Remark \ref{RmkState}.
\end{rmk}

We comment on the classical situation. As is shown in Example \ref{ExampleClassicalDual}, the Fourier transform is implicitly used to define the dual quantum group of a classical abelian group. It is for this reason that the $L^2$-Fourier transform $\mathcal{F}_2$ trivializes on the level of GNS-spaces. We work this out in the next example.
\begin{example}
Let $G$ be a locally compact abelian group. For $f \in L^1(G)$, let $\omega_f$ be the the normal functional on $L^\infty(G)$ given by $\omega_f(g) = \int_G f(x) g(x) d_lx$. Then,
\[
\lambda((\omega_f)) = (\omega_f \otimes \iota)(W) = \int_G f(x) \lambda_x d_lx \in \mathcal{L}(G),
\]
where $x \mapsto \lambda_x$ is the left-regular representation. On the other hand, using the direct integral decomposition $L^\infty(\hat{G}) = \int_{\hat{G}}^\oplus \mathbb{C} d\pi$, we find for (\ref{EqnIntroFT}),
\[
\hat{f} = \int_{\hat{G}}^\oplus \int_G f(x) \pi(x) dx \: d\pi =  \int_G   f(x) \int_{\hat{G}}^\oplus \pi(x) d\pi \: dx .
\]
The left regular representation $x \mapsto \lambda_x$ is unitarily equivalent to $\int^\oplus_{\hat{G}} \pi d\pi$, where the intertwiner is given by the (classical) $L^2$-Fourier transform. Since the dual quantum group associated to a classical group is given by conjugating $L^\infty(\hat{G})$ with the classical $L^2$-Fourier transform, see Example \ref{ExampleClassicalDual}, the indentification of the dual GNS-space $L^2(\hat{G})$ with the GNS-space $L^2(G) = \cH$ is given by applying the classical (inverse) $L^2$-Fourier transform.
Hence, using these identifactions, we see that the transform defined in (\ref{EqnIntroFT}) is the quantum group analogue of the transform of Theorem \ref{ThmFT}. 
\end{example}

Note that it is due to the identifications $L^2(M)_{{\rm left}}$ with $\cH$ and $L^2(\hat{M})_{{\rm left}}$ with $\hat{\cH} = \cH$ that the $L^2$-Fourier transform becomes the identity map. If we had not made these identifications the map would be less trivial. It is for this reason that we have choosen to write {\it unitary map} in the first statement of Theorem \ref{ThmFT} instead of {\it identity map}.   

\vspace{0.3cm}

Note that moreover, our transform coincides with the definition given in \cite[Definition 1.3]{DaeleFourier}. To comment on this, suppose that $(M, \Delta)$ is compact, i.e. $\varphi$ is a state. Let $A\subseteq M$ be the Hopf algebra of the underlying algebraic quantum group. We  mention that $A$ is the Hopf algebra of matrix coefficients of irreducible, unitary corepresentations of $M$ and refer to \cite{Tim} for more explanation. Let $\hat{A}$ be its dual, which is the space of linear functionals on $A$ of the form $\varphi( \: \cdot \: x)$, where $x \in A$, \cite[Theorem 1.2]{DaeleFourier}. 
Van Daele defines the transform by
\[
A \rightarrow \hat{A}: x \mapsto \varphi(\: \cdot \: x), \qquad x \in A.
\]
 
On the other hand, by Remark \ref{RmkState}, the normal functional $_x \varphi$, with $x \in L$ is given by $x\varphi$, since we assumed that $\varphi$ is a state. Here $x\varphi \in M_\ast$ is defined by $(x\varphi)(y) = \varphi(yx), y \in M$ (we use this notation to distinguish it from the algebraic map $\varphi(\: \cdot \: x)$). So the Fourier transform is defined by:
\[
x \mapsto \lambda(x \varphi), \qquad x \in M.
\]
Since in the transition of compact algebraic quantum groups to compact von Neumann algebraic quantum groups, the element $\varphi(\:\cdot\: x) \in \hat{A}$  corresponds to $\lambda(x \varphi) \in \hat{M}$, this shows the correspondence. 
Here, we refer to \cite{Tim} for compact algebraic quantum groups and their relations to locally compact quantum groups.

\vspace{0.3cm}

By Pontrjagin duality, one also find dual Fourier transforms. At every point in Theorem \ref{ThmFT} where one of the objects $L, R, M, \varphi, \Lambda, \lambda,  \mathcal{F}_2, \mathcal{F}_1$ appears, one should replace the object by the same object equipped with a hat and vise versa. In that way, we get a dual $L^2$-Fourier transform $\hat{\mathcal{F}}_2: \cH \rightarrow \cH $, determined by
\[
\hat{\Lambda}(x) \mapsto \Lambda(\hat{\lambda}(_x \hat{\varphi})), \qquad x \in \hat{L}.
\]
Since this map is on the GNS-level given by the identity, we automatically find the following corollary.
\begin{cor} \label{CorInverse}
We have $\mathcal{F}_2^{-1} = \hat{\mathcal{F}}_2$.
\end{cor}
  
Next, we apply the complex interpolation method to define $L^p$-Fourier transforms. 

\begin{thm}\label{ThmLpTransform}  Let $p \in [1,2]$ and set $q$ by $1/p+1/q = 1$.
There exists a unique bounded linear map $\mathcal{F}_p: L^p(M)_{{\rm left}} \rightarrow L^q(\hat{M})_{{\rm left}}$ such that $\mathcal{F}_p$ is compatible with $\mathcal{F}_1$ and $\mathcal{F}_2$ in the sense of Remark \ref{RmkCptMor}, i.e. the following diagram commutes:
\begin{equation}\label{EqnDiagLp}
\xymatrix{
&M_\ast \ar[dr]^{(\ref{EqnPairingMast})} \ar@2[rrrr]^{\mathcal{F}_1} &&&&\hat{M} \ar[dr]^{(\ref{EqnPairingM})} &\\
L \ar[rd]^{\Lambda}\ar[ru]^{l^1}\ar[r]^{l^p}  & L^p(M)_{{\rm left}} \ar[r]  \ar@2@(dr,dl)[rrrr]^{\mathcal{F}_p} & R^\ast & & \hat{L}\ar[ru]^{\hat{l}^1} \ar[r]^{\hat{l}^q}\ar[rd]_{\hat{\Lambda}} & L^q(\hat{M})_{{\rm left}} \ar[r]  & \hat{R}^\ast \\
& \cH \ar[ru]_{(\ref{EqnPairingGNS})}\ar@2[rrrr]^{\mathcal{F}_2} &&&& \cH \ar[ur]_{(\ref{EqnPairingGNS})} &\\
}
\end{equation}
Moreover, $\Vert \mathcal{F}_p \Vert \leq 1$.
\end{thm}
\begin{proof}
We can apply the complex interpolation method with parameter $\theta = 2/p-1 = 1-2/q$ to the pairs $(\cH, M_\ast)$ and $( \cH, \hat{M})$ which are  compatible couples as in (\ref{EqnDiagL1L2}). By Theorem \ref{ThmInterpolationL2} the corresponding interpolation spaces are respectively $L^p(M)_{{\rm left}}$ and $L^q(\hat{M})_{{\rm left}}$.

Since by Theorem \ref{ThmLpTransform}, $\mathcal{F}_1: M_\ast \rightarrow \hat{M}$ and $\mathcal{F}_2: \cH \rightarrow \cH$ are compatible, with respect to diagram (\ref{EqnDiagL1L2}), we can use Remark \ref{RmkCptMor} to obtain a map $\mathcal{F}_p: L^p(M)_{{\rm left}} \rightarrow L^q(\hat{M})_{{\rm left}}$ with the desired properties.
\end{proof}

We conclude this section by giving the Fourier transform explicitly in terms of Hilsum's $L^p$-spaces.  We omit the proof and merely give a few comments. The result relies on some technicalities involving Hilsum's $L^p$-spaces, which was not our focus. The theorem is not needed for the the subsequent sections.

\begin{thm}\label{ThmLpTransformHilsum}
Let $p \in [1,2]$ and set $q$ by $1/p+1/q=1$. Fix a normal, semi-finite, faithful weight $\phi$ on $M'$ and $\hat{\phi}$ on $\hat{M}'$. Set the corresponding spatial derivatives $d = d\varphi/d\phi$ and $\hat{d} = d\hat{\varphi}/d\hat{\phi}$. 
Then,
\[
\hat{\Phi}_q^{-1} \mathcal{F}_p \Phi_p: L^p(\phi) \rightarrow L^q(\hat{\phi}): [a d^{1/p}] \mapsto [\lambda(_{a } \varphi) \hat{d}^{1/q}], \qquad a \in \mathcal{T}_\varphi^2.
\]
\end{thm}
Note that in Theorem \ref{ThmLpTransformHilsum}, we see that $[a d^{1/p}]$ is in $L^p(\phi)$ by Proposition \ref{PropL2Identifications}. Moreover, since $ _{a } \varphi  \in \mathcal{I}$, we see that $\lambda(_{a } \varphi) \in \mathfrak{n}_{\hat{\varphi}}$. Therefore, $ [\lambda(_{a } \varphi) \hat{d}^{1/q}]$ is in $L^q(\hat{\phi})$ by Lemma \ref{LemInc}. 
The theorem follows by a careful analyis of (\ref{EqnDiagLp}) involving Proposition \ref{PropL2Identifications}. The proof then relies on the following fact. For $x \in \mathfrak{n}_{\hat{\varphi}}$, one can consider $x$ as an element of $L^q(\hat{M})_{{\rm left}}$, see Section \ref{SectIntersections}, and one can prove that $\hat{\Phi}_q x = [x\hat{d}^{1/q}]$.

\section{Convolution product}\label{SectApplications}

 We   define convolutions of elements in $L^1(M)_{{\rm left}} = M_\ast$ with elements in $L^p(M)_{{\rm left}}$. We prove that the Fourier transform transfers the convolution product into a product on the dual quantum group.

\begin{nt}
We keep the notation as in Section \ref{SectFT}, c.f. Notation \ref{NtSect5}.
\end{nt}

 Note that since Izumi's $L^p$-spaces are defined by means of complex interpolation, there is a priori no multiplication on these spaces. 
Therefore, we  extend the multiplication of $M$ to the $L^p$-setting in the following proposition. This seems to be the most natural definition of a multiplication in the $L^p$-setting. The proof of the following proposition is completely similar to the one of Theorem \ref{ThmFT} (\ref{ItemFTIII}) and \ref{ThmLpTransform}.
 
\begin{prop}\label{PropProduct} We extend the product of $M$ to the $L^p$-setting.
\begin{enumerate}
\item Let $x \in M$. The maps 
\[
\begin{split}
m_x^{\infty}: M \rightarrow M: y \mapsto xy, \\
m_x^1: M_\ast \rightarrow M_\ast: \omega \mapsto x \omega,
\end{split}
\]
 are compatible in the sense of Remark \ref{RmkCptMor}, i.e. the non-dotted arrows in the following diagram commute:
\begin{equation}\label{EqnDiagramProduct}
\xymatrix{
&M_\ast \ar[dr]^{(\ref{EqnPairingMast})} \ar@2[rrrr]^{m_x^1} &&&&M_\ast \ar[dr]^{(\ref{EqnPairingMast})} &\\
L \ar[rd]^{l^\infty}\ar[ru]^{l^1}\ar[r]^{l^p}  & L^p(M)_{{\rm left}} \ar[r]  \ar@2@{==>}@(dr,dl)[rrrr]^{m_x^p} & R^\ast & & L \ar[ru]^{l^1} \ar[r]^{l^p}\ar[rd]_{l^\infty} & L^p(M)_{{\rm left}} \ar[r]  & R^\ast .\\
& M \ar[ru]_{(\ref{EqnPairingM})}\ar@2[rrrr]^{m_x^\infty} &&&& M\ar[ur]_{(\ref{EqnPairingM})} &\\
}
\end{equation}
\item Let $p \in (1, \infty)$. There is a unique bounded map $m_x^p: L^p(M)_{{\rm left}} \rightarrow L^p(M)_{{\rm left}}$ that is compatible with $m_x^\infty$ and $m_x^1$, i.e. the dotted arrow in (\ref{EqnDiagramProduct}) makes the diagram commutative. 
\end{enumerate}
\end{prop}

\begin{dfn}\label{DfnProduct}
Let $x \in M$ and let $y \in L^p(M)_{{\rm left}}$.  We will write $xy$ for $m_x^p(y)$.
\end{dfn}

For $\omega_1, \omega_2 \in M_\ast$, we define the convolution product,
\[
\omega_1 \ast \omega_2 = (\omega_1 \otimes \omega_2) \circ \Delta.
\]
 This product is well-known in the theory of l.c. quantum groups. We show that it is possible to extend it to the $L^p$-setting for $p \in [1,2]$. 
Moreover, the convolution product is turned into the product of Definition \ref{DfnProduct} by the Fourier transfrom. 

\begin{thm} Let $p \in [1,2]$ and set $q \in [2, \infty]$ by $1/p+1/q = 1$.
\begin{enumerate}
\item\label{ItemConvoI} Let $x \in L$ and let $\omega \in M_\ast$. Then, 
\[
\omega \ast (_x\varphi)    \in \mathcal{I} \qquad {\rm  and  } \qquad  \xi(\omega \ast (_x \varphi) ) = \lambda(\omega) \Lambda(x).
\]
\item\label{ItemConvoII}  Let $\omega \in M_\ast$. We denote $\omega \ast^2$ for the bounded operator  $\lambda(\omega): \cH \rightarrow \cH$. Furthermore, we define  $\omega \ast^1: M_\ast \rightarrow M_\ast: \theta \mapsto \omega \ast \theta$. Then, $\omega \ast^1$ and $\omega \ast^2$ are compatible, i.e. the non-dotted arrows in following diagram commute:
 \begin{equation}\label{EqnDiagramConvolution}
\xymatrix{
&M_\ast \ar[dr]^{(\ref{EqnPairingMast})} \ar@2[rrrr]^{\omega \ast^1} &&&&M_\ast \ar[dr]^{(\ref{EqnPairingMast})} &\\
L \ar[rd]^{\Lambda}\ar[ru]^{l^1}\ar[r]^{l^p}  & L^p(M)_{{\rm left}} \ar[r]  \ar@2@{==>}@(dr,dl)[rrrr]^{\omega \ast^p} & R^\ast & & L \ar[ru]^{l^1} \ar[r]^{l^p}\ar[rd]_{\Lambda} & L^p(M)_{{\rm left}} \ar[r]  & R^\ast. \\
& \cH \ar[ru]_{(\ref{EqnPairingGNS})}\ar@2[rrrr]^{\omega \ast^2} &&&& \cH \ar[ur]_{(\ref{EqnPairingGNS})} &\\
}
\end{equation}
\item\label{ItemConvoIII}  There is a unique bounded operator $\omega \ast^p: L^p(M)_{{\rm left}}\rightarrow  L^p(M)_{{\rm left}} $ that is compatible with $\omega \ast^1$ and $\omega \ast^2$, i.e. (\ref{EqnDiagramConvolution}) commutes.
 \item\label{ItemConvoIV}  For $\omega \in M_\ast$, $a \in L^p(M)_{{\rm left}}$,
 \[
\mathcal{F}_1(\omega) \mathcal{F}_p( a ) = \mathcal{F}_p(\omega \ast^p a),
\]
where the left hand side uses Definition \ref{DfnProduct} for $L^q(\hat{M})_{{\rm left}}$.
\end{enumerate}
\end{thm}

\begin{proof}
Let $\theta \in \hat{\mathcal{I}}$ and put $y = \hat{\lambda}(\theta) = (\iota \otimes \theta)(W^\ast)$. Now, (\ref{ItemConvoII}) follows from,
\[
 \begin{split}
  &(\omega \ast (_x \varphi) )(y^\ast) = (\omega \otimes (_x \varphi) )  \Delta((\iota \otimes \theta)(W^\ast)^\ast)
= (\omega \otimes  (_x \varphi)  \otimes \overline{\theta})(W_{13} W_{23}) \\ =& \overline{\theta (\left( \: (\omega \otimes \iota)(W) ( (_x \varphi)  \otimes \iota)(W) \: \right)^\ast)} = \langle \hat{\Lambda}\left( (\omega \otimes \iota)(W) ( (_x \varphi)  \otimes \iota)(W) \right), \hat{\xi}(\theta) \rangle \\
= &\langle  (\omega \otimes \iota)(W) \xi(_x \varphi) , \Lambda(y) \rangle = \langle  \lambda(\omega) \Lambda(x), \Lambda(y) \rangle ,
 \end{split}
\]
and the fact that $\{ \Lambda((\iota \otimes \theta)(W^\ast)) \mid \theta \in \hat{\mathcal{I}} \}$ is dense in $\cH$. 

The compatibility in (\ref{ItemConvoII}) follows directly from (\ref{ItemConvoI}) using Theorem \ref{ThmIntersectionII}.
(\ref{ItemConvoIII}) follows by applying Theorem \ref{ThmInterpolationL2} to (\ref{ItemConvoII}).
(\ref{ItemConvoIV}) For $\omega_1, \omega_2 \in M_\ast$, note that
\[
\mathcal{F}_1(\omega_1 \ast \omega_2) = (\omega_1 \otimes \omega_2 \otimes \iota) (\Delta \otimes \iota) (W) = (\omega_1 \otimes \omega_2 \otimes \iota) W_{13} W_{23} = (\omega_1 \otimes \iota)(W) (\omega_2 \otimes \iota)(W).
\]
For $x \in L, \omega \in M_\ast$, 
\[
  \mathcal{F}_p(\omega \ast^p l^p(x) )  =  \mathcal{F}_1(\omega \ast (_x \varphi))   =    \mathcal{F}_1(\omega)  \mathcal{F}_1(_x\varphi) =   \mathcal{F}_1(\omega)  \mathcal{F}_p( l^p(x)). 
\]
Here, the first and last equality follows from commutativity of (\ref{EqnDiagLp}), (\ref{EqnDiagramProduct}) and (\ref{EqnDiagramConvolution}).
Since the range of $l^p$ is dense in $L^p(M)_{{\rm left}}$, see Lemma \ref{LemDenseInt},  (\ref{ItemConvoIV}) follows.

\end{proof}

\section{A distinguished choice for the interpolation parameter}\label{SectSUTwo}
 
 Recall that in Sections \ref{SectIntersections} to  \ref{SectApplications} we considered the compatible couple $(M, M_\ast)$ for the interpolation parameter $z = -1/2$, see Definition \ref{DfnL}. For this parameter one is able to define a $L^p$-Fourier transform. In this section we show that the real part of the parameter is distinguished. 
More precisely, we investigate the example of $(M, \Delta) = SU_q(2)$ and show that given the fact that 
\begin{equation}\label{EqnLambda}
\mathcal{F}_1: M_\ast \rightarrow \hat{M}: \omega \mapsto (\omega \otimes \iota)(W),
\end{equation}
 is the $L^1$-Fourier transform the only interpolation parameters $z$ that allows a passage to a $L^p$-Fourier transform are $z = -1/2 + it$, where $t \in \mathbb{R}$. 

\vspace{0.3cm}

The importance of this result is strengthened by the final remark of \cite{DawsRunde}. For classical, locally compact groups there is an approximation property called Reiter's property $(P_p)$, where $p \in [1, \infty)$. The definition assumes the existence of a net of functions in $L^p(G)$ satisfying the approximation axiom of \cite[Definition 1.2]{DawsRunde}. Daws and Runde show that $(P_1)$ and $(P_2)$ can be defined for quantum groups as well and they use them to study (co-)amenability properties of quantum groups. 
 
In the final remark of \cite{DawsRunde}, Daws and Runde mention that it remains to be seen if there is a property $(P_p)$ for any $p \in [1, \infty)$. In particular, they mention that it remains unclear how the $L^p$-space associated with a quantum group should be turned into a $L^1$-module. In \cite{ForLeeSam} this is done using Izumi's $L^p$-spaces for the complex interpolation parameter $z = -1/2$, whereas in \cite{Daws} a similar, but not identical construction was used for the parameter $z = 0$. 
We believe that the Fourier transform indicates that the most natural choice would be $z = -1/2$. 

\vspace{0.3cm}
 
From now on we let $(M, \Delta)$ be the quantum group $SU_q(2)$, see \cite{Wor}, \cite{WorTwisted}, \cite{KliSch}. See also \cite{KusLec} for a concise introduction.  We recall its most important properties.  
 
We set $\cH = L^2(\mathbb{N}) \otimes L^2(\mathbb{T})$. Let $(e_i)_{i \in \mathbb{N}}$ be the canonical orthonormal basis of $L^2(\mathbb{N})$ and let $(f_k)_{k \in \mathbb{Z}}$ be the canonical orthonormal basis for $L^2(\mathbb{Z})$ (so $f_k = \zeta^k$, where $\zeta$ is the identity function on the complex unit circle $\mathbb{T}$). 
  Define operators $\alpha, \gamma$ given by:
\begin{equation}\label{EqnAlphaGamma}
\alpha \: e_i \otimes f_k = \sqrt{1-q^{2i}} e_{i-1} \otimes f_k, \qquad 
\gamma \: e_i \otimes f_k = q^i e_{i} \otimes f_{k+1}.
\end{equation}
%Then, $\alpha, \gamma$ satisfy the well-known relations:
%\[
% \alpha^\ast \alpha + \gamma^\ast \gamma = 1, \: \alpha \alpha^\ast + q^2 \gamma \gamma^\ast = 1,\:
%\gamma \gamma^\ast = \gamma^\ast \gamma, \: q \gamma \alpha = \alpha \gamma, \: q \gamma^\ast \alpha = \alpha \gamma^\ast.
%\]
Then, $M$ is the von Neumann algebra generated by $\alpha$ and $\gamma$, i.e. 
\[
M = B(L^2(\mathbb{N})) \otimes L^\infty(\mathbb{T}) \simeq  L^\infty(\mathbb{T}, B(L^2(\mathbb{N}))).
\]
 For $x = x(t) \in L^\infty(\mathbb{T}, B(L^2(\mathbb{N})))$, both the left and right Haar weight are given by the state
\[
\varphi(x) =  \frac{(1-q^2)}{2\pi} \int_{\mathbb{T}} \sum_{i=0}^\infty q^{2i}   \langle x(t) e_{i}, e_i \rangle dt.
\]
%The coproduct of $M$ is given by the unique normal $\ast$-homomorphism extending
%\[
%\Delta(\alpha) = \alpha \otimes \alpha - q \gamma^\ast \otimes \gamma,\quad 
% \Delta(\gamma) = \gamma \otimes \alpha - \alpha^\ast \otimes \gamma.
%\]

Next, we need Peter-Weyl theory for $SU_q(2)$. 
Recall \cite{KliSch}  that for every $l \in \frac{1}{2} \mathbb{N}$, there exists a unique irreducible corepresentation $t^{(l)} \in M \otimes M_{2l+1}(\mathbb{C})$.
In fact, these are all the irreducible corepresentations of $(M, \Delta)$ and we have a Peter-Weyl decomposition
\[
W \simeq \bigoplus_{l \in\frac{1}{2} \mathbb{N}}  t^{(l)} \otimes 1_{2l+1} \quad  (\quad  \in    M \otimes \bigoplus_{l \in \mathbb{N}} M_{2l+1}(\mathbb{C}) \otimes M_{2l+1}(\mathbb{C}) \quad ).
\]
So every corepresentation $t^{(l)}$ appears $2l+1$ times in the multiplicative unitary.
 
Let $g^{(l)}_{-l}, g^{(l)}_{-l+1}, \ldots, g^{(l)}_l$ denote the standard basis vectors of $\mathbb{C}^{2l+1}$. Let us denote $t^{(l)}_{i,j}$ for the matrix elements $(\iota \otimes \omega_{g^{(l)}_j, g^{(l)}_i})(t^{(l)})$. For every $l\in \frac{1}{2}\mathbb{N}$, there exists a unique strictly positive operator $Q^{(l)}$ such that we have orthogonality relations between matrix coefficients.
\begin{equation}\label{EqnOrthogonality}
\varphi((t^{(l)}_{i,j})^\ast t^{(l')}_{i',j'}) = \delta_{l, l'} \delta_{ j, j'} \langle Q^{(l)} g^{(l)}_i, g^{(l)}_{i'}\rangle 
\end{equation}
In fact, with respect to the basis $g^{(l)}_i$, the matrix $Q^{(l)}$ is diagonal. It follows that $\hat{M} \simeq \oplus_{l \in \frac{1}{2} \mathbb{N}} M_{2l+1}(\mathbb{C})$. We put $Q = \oplus_{l \in \frac{1}{2} \mathbb{N}}  Q^{(l)}$, so that $Q$ is affiliated with $\hat{M}$. Moreover, 
\begin{equation}\label{EqnModularQ}
\hat{\sigma_t}(x) = Q^{-it} x Q^{it}, \qquad x \in \hat{M}.
\end{equation}
 
Finally, the following two matrix coefficients will play an essential role in the proof of the main theorem of this section. It follows from \cite[Chapter 4]{KliSch} that for $n \in \mathbb{N}$,
\[
 t^{(n/2)}_{n/2, n/2}  = \alpha^n, \qquad 
 t^{(n/2)}_{-n/2, -n/2}  = (\alpha^\ast)^n.
\]
 
Recall the notational conventions from Section \ref{SectLp}.  

\begin{thm}
Consider $(M, \Delta) = SU_q(2)$ and let $z, z' \in \mathbb{C}$. Let $\mathcal{F}_1: M_\ast \rightarrow \hat{M}$ be defined as in (\ref{EqnLambda}). Suppose that there is bounded map $F_2: L^2_{(z)}(M) \rightarrow L^2_{(z')}(\hat{M})$ making the following diagram commutative
\begin{equation}\label{EqnExistence}
\xymatrix{
&M_\ast \ar[dr]^{(i_{(-z)}^\infty)^\ast} \ar@2[rrrr]^{\mathcal{F}_1} &&&&\hat{M} \ar[dr]^{(\hat{i}_{(-z')}^1)^\ast} &\\
L_{(z)} \ar[r]^{i_{(z)}^2}\ar[ru]^{i_{(z)}^1} & L^2_{(z)}(M) \ar[r]^{\subseteq} \ar@2@(dr,dl)[rrrr]_{F_2} & L_{(-z)}^\ast & & \hat{L}_{(z')}\ar[ru]^{_{\hat{i}_{(z')}^\infty}} \ar[r]^{_{\hat{i}_{(z')}^2}} & L^2_{(z')}(\hat{M}) \ar[r]^{\subseteq}  &  \hat{L}_{(-z')}^\ast.
}
\end{equation}
Then, $z = -1/2 + it$  for some $t \in \mathbb{R}$. 
\end{thm}
\begin{proof}
We will prove that $F_2$ is unbounded unless $z = -1/2 + it$  for some $t \in \mathbb{R}$. We need three preparations.

Firstly, the modular automorphism group of $\varphi$ is given by $\sigma_t(x) = ( \gamma\gamma^\ast)^{it} x  (\gamma\gamma^\ast)^{-it}$. Hence, it follows from (\ref{EqnAlphaGamma}) that $\alpha  \in \mathcal{T}_\varphi$ and 
\[
\sigma_z(\alpha) = q^{-2iz} \alpha, \qquad \sigma_z(\alpha^\ast) = q^{2iz} \alpha^\ast.
\]
 
Secondly, for $a,b \in \mathcal{T}_\varphi, x \in \mathcal{T}_\varphi^2, z \in \mathbb{C}$, we find 
\[
\begin{split}
 \varphi^{(z)}_x (a^\ast b)  =& \langle x J \nabla^{\bar{z}} \Lambda(a) , J \nabla^{-z} \Lambda(b) \rangle
=
 \langle \nabla^{z+\frac{1}{2}} x  \nabla^{-z-\frac{1}{2}} \nabla J \nabla^{1/2} \Lambda( a) ,   J \nabla^{1/2}\Lambda(b) \rangle
\\ = &  \langle  \sigma_{-i(z+\frac{1}{2})} (x)  \Lambda(\sigma_{-i}(a^\ast)) , \Lambda(b^\ast) \rangle = \varphi ( b \sigma_{-i(z+\frac{1}{2})} (x)\sigma_{-i}(a^\ast))
  = \varphi (a^\ast b \sigma_{-i(z+\frac{1}{2})} (x)).
\end{split}
\]
So we conclude that  $\varphi^{(z)}_x =  \sigma_{-i(z+\frac{1}{2})} (x) \varphi$.

%Note that by the Peter-Weyl decomposition (\ref{EqnPeterWeyl}), we see that \[
%\hat{M} = \overline{\{ (\omega \otimes \iota)(W) \mid \omega \in M_\ast \}}^{\sigma{\rm -strong-} \ast}
%\]
% is a direct sum of matrix algebras. Hence, $\hat{\varphi}$ has a Radon-Nikodym derivative with respect to a trace, which can be choosen to be $Q^{-1}$.  So for $x \in \hat{M}$, we have:
%\begin{equation}\label{EqnModularQ}
%\hat{\sigma}_t(x) = Q^{-it} x Q^{it}.
%\end{equation}

Thirdly, we identify $\hat{M}$ with $\oplus_{l \in \frac{1}{2}\mathbb{N}} M_{2l+1}(\mathbb{C})$. Let $e^{(n/2)}_{n/2,n/2}$ (and $e^{(n/2)}_{-n/2,-n/2}$ ) be the element of $\hat{M}$, with matrix elements equal to zero everywhere, except for the summand with index $n/2$, where it has a 1 on the upper left (respectively lower right) corner. 
 Using the Peter-Weyl orthogonality relations, we see that for every $n \in \mathbb{N}$:
\begin{equation}\label{EqnFourierAlpha}
\begin{split}
(   \varphi \otimes \iota)(W (\alpha^n \otimes 1)) =
\bigoplus_{l \in \frac{1}{2}\mathbb{N}}  (\varphi \otimes \iota) (t^{(l)}  (t^{(n/2) \ast}_{-n/2, -n/2}  \otimes 1))  =
 & \varphi((\alpha^\ast)^n \alpha^n)   e^{(n/2)}_{-n/2,-n/2},\\
  (   \varphi \otimes \iota)(W ((\alpha^\ast)^n \otimes 1) ) = 
\bigoplus_{l \in \mathbb{N}}  (\varphi \otimes \iota)( t^{(l)}  ( t^{(n/2)\ast}_{n/2, n/2}  \otimes 1)) =
& \varphi( \alpha^n(\alpha^\ast)^n)  e^{(n/2)}_{n/2,n/2},
\end{split}
\end{equation}
\[
 \varphi((\alpha^\ast)^n \alpha^n)  = \frac{(1-q^2)q^{2n}}{1-q^{2n+2}}, \qquad 
 \varphi( \alpha^n(\alpha^\ast)^n) = \frac{(1-q^2)}{1-q^{2n+2}} .
\]
By (\ref{EqnModularQ}) and the fact that $Q$ is diagonal we see that  (\ref{EqnFourierAlpha}) are in $\mathcal{T}_{\hat{\varphi}}^2$ and 
\begin{equation}\label{EqnDualModular}
\hat{\sigma}_z ( e^{(n/2)}_{-n/2,-n/2} ) = e^{(n/2)}_{-n/2,-n/2}, \qquad \hat{\sigma}_z ( e^{(n/2)}_{n/2,n/2} ) = e^{(n/2)}_{n/2,n/2} .
\end{equation}

Now we prove that $F_2$ must be unbounded by proving that the map $ U_{(z')} F_2 U_{(z)}^\ast : \cH \rightarrow \cH$ is unbounded. Recall that $U_{(z)}$ was defined in Proposition \ref{PropL2Identifications}.
\[
\begin{split}
 & F_2 U_{(z)}^\ast \Lambda(\sigma_{-i(z+1/2)/2}(\alpha^n) ) =
  F_2   i_{(z)}^2( \alpha^n  ) = 
  (\hat{i}_{(-z')}^1)^\ast F_1   i_{(z)}^1( \alpha^n  ) \\= &
 (\hat{i}_{(-z')}^1)^\ast( \varphi^{(z)}_{\alpha^n} \otimes \iota)(W) = 
(\hat{i}_{(-z')}^1)^\ast ( \varphi \otimes \iota)(W  (\sigma_{-i(z+1/2)} (\alpha^n)  \otimes 1)) \\ = & 
  q^{-2n(z+1/2)} (\hat{i}_{(-z')}^1)^\ast ( \varphi \otimes \iota)(W  (\alpha^n \otimes 1)) .
\end{split}
\]
Since $ ( \varphi \otimes \iota)(W  (\alpha^n \otimes 1)) \in \mathcal{T}_{\hat{\varphi}}^2$, we see that by commutativity of the right triangle in (\ref{EqnExistence}), 
\[
 F_2 U_{(z)}^\ast \Lambda(\sigma_{-i(z+1/2)/2}(\alpha^n) ) = 
 q^{-2n(z+1/2)} (\hat{i}_{(z')}^2)  ( \varphi \otimes \iota)(W  (\alpha^n \otimes 1)).
\]
Hence, 
\[
\begin{split}
&  U_{(z')} F_2 U_{(z)}^\ast \Lambda(\sigma_{-i(z+1/2)/2}(\alpha^n) ) = 
 q^{-2n(z+1/2)}  \hat{\Lambda}(\hat{\sigma}_{-i (z'+1/2)/2} ( \varphi \otimes \iota)(W  (\alpha^n \otimes 1))) \\ =\!\!\!\!\!\!\!\!\!\!\!\!\!\!^{(\ref{EqnFourierAlpha}), (\ref{EqnDualModular})} &  
 q^{-2n(z+1/2)}  \hat{\Lambda}( ( \varphi \otimes \iota)(W  (\alpha^n \otimes 1))) =  
 q^{-2n(z+1/2)} \xi(\alpha^n \varphi) =  
q^{-2n(z+1/2)} \Lambda( \alpha^n  ).
\end{split}
\]
Hence,
\[
 U_{(z')} F_2 U_{(z)}^\ast:  \Lambda(\alpha^n) = q^{n(z+1/2)} \Lambda(\sigma_{-i(z+1/2)/2}(\alpha^n) ) \mapsto q^{-n(z+1/2)} \Lambda(\alpha^n). 
\]
which is unbounded in case ${\rm Re}(z) > -1/2$.

 By a similar computation, one finds that
\[
 U_{(z')} F_2 U_{(z)}^\ast:  \Lambda((\alpha^\ast)^n) \mapsto q^{n(z+1/2)} \Lambda((\alpha^\ast)^n).
\]
In this case we see that  $U_{(z')} F_2 U_{(z)}^\ast$ is unbounded for ${\rm Re}(z) < -1/2$. 
\end{proof}

\begin{rmk}
By Pontrjagin duality also the dual statement holds. So in order to get a proper Fourier theory on quantum groups with Fourier transforms and inverse Fourier transforms, one is obliged to take the interpolation parameter on both $M$ and $\hat{M}$ to be $-1/2$.
\end{rmk}

\appendix

\section{}
We have not found an explicit proof of the following lemma's  in the literature. For completeness, we prove them here. Here $M$ is a von Neumann algebra with normal, semi-finite, faithful weight $\varphi$.

The following lemma is a variant of \cite[Lemma 9]{TerpII}. 

\begin{lem}\label{LemApprox}
 Let $\delta >0$. There exists a  net $(e_j)_{j \in J}$ in $\mathcal{T}_\varphi$ such that {\rm (1)} $\Vert \sigma_z(e_j) \Vert \leq e^{\delta {\rm Im}(z)^2}$,  {\rm (2)} $e_j \rightarrow 1$ strongly and  {\rm (3)}  $\sigma_{i/2}(e_j) \rightarrow 1$ $\sigma$-weakly. 
\end{lem}
\begin{proof}
Let $(f_j)_{j \in J}$ and $(e_j)_{j \in J}$ be nets as in \cite[Lemma 9]{TerpII}. This lemma proves already that $(e_j)_{j \in J}$ satisfies  (1) and (2). Now, (3) follows, since for $\xi \in\cH$,
\[
\begin{split}
 \langle \sigma_{\frac{i}{2}}(e_j) \xi, \xi  \rangle = &  \omega_{\xi, \xi}\left( \sqrt{\frac{\delta}{\pi}} \int_{-\infty}^\infty e^{-\delta(t - \frac{i}{2})^2} \sigma_t(f_j) dt \right)
 =  \left( \sqrt{\frac{\delta}{\pi}} \int_{-\infty}^\infty e^{-\delta(t - \frac{i}{2})^2} (\omega_{\xi, \xi} \circ \sigma_t ) dt \right)( f_j) \\ \rightarrow & \left( \sqrt{\frac{\delta}{\pi}} \int_{-\infty}^\infty e^{-\delta(t - \frac{i}{2})^2} (\omega_{\xi, \xi} \circ \sigma_t) dt\right) ( 1) =  \sqrt{\frac{\delta}{\pi}} \int_{-\infty}^\infty e^{-\delta(t - \frac{i}{2})^2} dt \langle \xi, \xi \rangle = \langle \xi, \xi \rangle,
\end{split}
\]
where the last equality is obtained by means of the residue
formula for meromorphic functions. So
$\sigma^\varphi_{\frac{i}{2}}(e_k)$ is bounded and converges
weakly, hence $\sigma$-weakly to 1.
\end{proof}

\begin{lem}\label{LemCoreT}
$\mathcal{T}_\varphi^2$ is a $\sigma$-strong-$\ast$/norm core for $\Lambda$.
\end{lem}
\begin{proof}
It is enough to prove that $\mathcal{T}_\varphi^2$ is a $\sigma$-weak/weak core for $\Lambda$, since the $\sigma$-weak/weak continuous functionals on the graph of $\Lambda$ equal the $\sigma$-strong-$\ast$/norm continuous functionals. It is not too hard to show that $\nphi \cap \nphi^\ast$ is a $\sigma$-weak/weak core for $\Lambda$. Now, let $x \in \nphi \cap \nphi^\ast$.  Put 
\[
x_n = \frac{n}{\sqrt{\pi}} \int_{-\infty}^{\infty} e^{-(nt)^2} \sigma_t(x) dt,
\]
where the integral is taken in the $\sigma$-strong-$\ast$ sense. By standard techniques (c.f. the proof of \cite[Lemma 9]{TerpII}), $x_n \in \mathcal{T}_\varphi$ and $x_n$ converges $\sigma$-weakly to $x$. Moreover, using the the fact that $\Lambda$ is  $\sigma$-strong-$\ast$/norm closed, we obtain
\[
\Lambda(x_n) =  \frac{n}{\sqrt{\pi}} \int_{-\infty}^{\infty} e^{-(nt)^2} \nabla^{it} \Lambda(x) dt \rightarrow \Lambda(x) \qquad {\rm weakly},
\]
where the integral is a Bochner integral, c.f. \cite[Chapter VI, Lemma 2.4]{TakII}. Hence $\mathcal{T}_\varphi$ is a core for $\Lambda$. Now, let $x \in \mathcal{T}_\varphi$ and let $(e_j)_{j \in J}$ be a net in $\mathcal{T}_\varphi$ such that $e_j \rightarrow 1$ $\sigma$-weakly, c.f. Lemma \ref{LemApprox}. Then, $e_jx \in \mathcal{T}_\varphi^2$ and $e_j x \rightarrow x$ $\sigma$-weakly and $\Lambda(e_jx) = e_j \Lambda(x) \rightarrow \Lambda(x)$ weakly. 
\end{proof}

\subsection*{Acknowledgement}

The author thanks Erik Koelink for useful discussions and also
many thanks to the referee for his comments and suggestions.

\end{document}